\documentclass[11pt]{article} 
\pdfoutput=1
\usepackage[utf8]{inputenc}
\usepackage[T3,T1]{fontenc}
\usepackage[english]{babel}

\usepackage{amsmath,amsfonts,amssymb,amsthm}
\usepackage{enumitem}
\usepackage{mathtools}
\usepackage{booktabs}

\DeclareSymbolFont{tipa}{T3}{cmr}{m}{n}
\DeclareMathAccent{\invbreve}{\mathalpha}{tipa}{16}

\usepackage[mono=false]{libertine}
\usepackage[cmintegrals,libertine]{newtxmath}
\usepackage[cal=euler, scr=boondoxo]{mathalfa}

\useosf
\usepackage[a4paper,vmargin={3.5cm,3.5cm},hmargin={2.5cm,2.5cm}]{geometry}
\usepackage[font={small,sf}, labelfont={sf,bf}, margin=1cm]{caption}
\usepackage[colorlinks=true]{hyperref}
\usepackage[usenames,dvipsnames]{color}
\usepackage{graphicx}

\usepackage{stackrel}
\usepackage{soul} 

\newcommand{\cev}[1]{\reflectbox{\ensuremath{\vec{\reflectbox{\ensuremath{#1}}}}}}

\theoremstyle{plain}
\newtheorem{theorem}{Theorem}
\newtheorem{corollary}[theorem]{Corollary}
\newtheorem{proposition}[theorem]{Proposition}
\newtheorem{lemma}[theorem]{Lemma}
\theoremstyle{definition}

\newcommand{\R}{\mathbb{R}}

\newcommand{\Z}{\mathbb{Z}}

\newcommand{\rmd}{\mathrm{d}}

\begin{document}
\title{\vspace{-1cm}\bf Discrete flat disks: rigid quadrangulations}
\author{\textsc{Timothy Budd}\footnote{IMAPP, Radboud University, Nijmegen, The Netherlands. Email: \href{mailto:t.budd@science.ru.nl}{t.budd@science.ru.nl}}}
\maketitle

\begin{abstract}
    Inspired by a question of Ferrari in the physics context of JT gravity, we introduce and enumerate a combinatorial family of quadrangulations of the disk, called \emph{rigid quadrangulations}.
    These form a subclass of the flat quadrangulations in the sense that every inner vertex has degree $4$, and therefore it can be viewed as a discrete model of flat metrics on the disk. 
    Our main result is a bijection between rigid quadrangulations and certain colorful $\Z$-labeled quadrangulations of the sphere, together with a dictionary relating a variety of natural statistics on both sides.
    Adaptions of the bijection to various boundary conditions allow us to import recent enumerative results for colorful quadrangulation obtained by Bousquet-Mélou and Elvey Price.
    We discuss some consequences of the enumeration of rigid quadrangulations for a flat version of JT gravity at finite cutoff, and comment on potential scaling limits.
\end{abstract}

\begin{figure}[h]
    \centering
    \includegraphics[width=.9\linewidth]{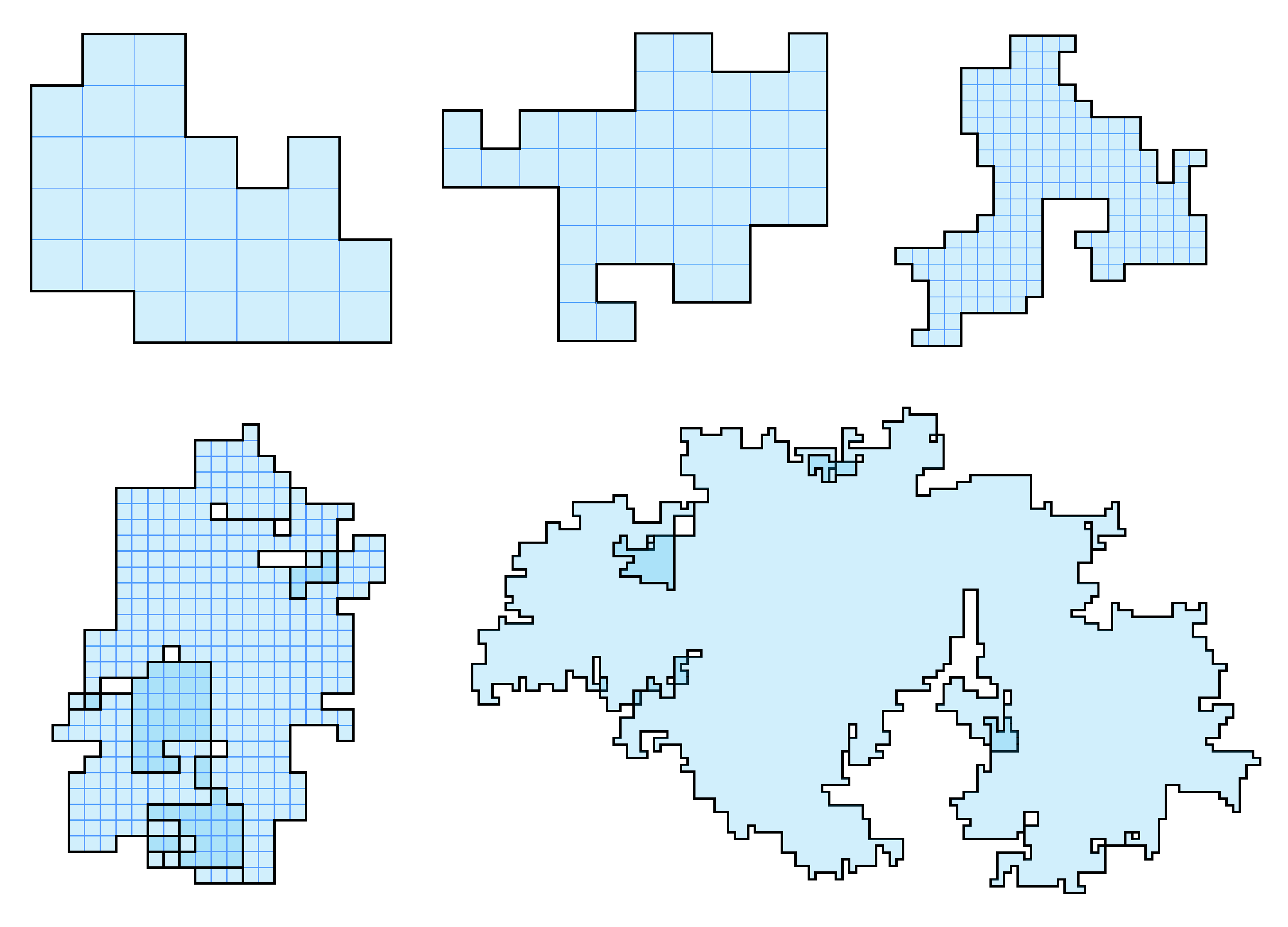}
    \caption{Uniform random rigid quadrangulations of fixed perimeter immersed in the plane.\label{fig:uniformrigid}}
\end{figure}

\tableofcontents

\section{Introduction}

The combinatorial study of planar maps has a long history, dating back to the works of Tutte in the 1960s \cite{Tutte_census_1963}.
The enumeration of many classes of planar maps has been achieved in the years since, displaying a great deal of universality regarding the asymptotics.
This is reflected in the universality of the scaling limit of the metric structure of large random planar maps in terms of the Brownian sphere \cite{LeGall_Uniqueness_2013}, providing a rigorous construction of two-dimensional quantum gravity envisioned in the physics literature.
To change the universality class one may consider maps with carefully chosen additional structure, that physicists would interpret as introducing critical matter systems coupled to gravity.
A more radical way to change the large scale geometry is to impose stringent curvature constraints by not only controlling the face degrees, but also the vertex degrees.
Perhaps the simplest example of this is the case of flat quadrangulations, where both the vertices and faces are required to have degree $4$.
Of course, there is only one such quadrangulation in the planar case, namely the infinite square lattice, while flat quadrangulations of the torus can be classified with some effort (see e.g.\ \cite{Pardo_Square_2023}).
Allowing for several vertices of different degree in a controlled fashion, the enumeration becomes a lot more interesting and brings one into the realm of square-tiled surfaces, origami, translation surfaces and Teichmüller dynamics \cite{Zorich_Flat_2006}.
This regime has also been studied in the physics literature in the context of dually-weighted matrix models \cite{Kazakov_Almost_1996,Kazakov_Disc_2022}. 
This work, instead, deals with flat quadrangulations with a boundary, particularly  the case of quadrangulations of the disk.
Away from the boundary these are still locally isomorphic to the square lattice, so the degrees of freedom are almost entirely determined by the geometry of the boundary.

This combinatorial analysis is inspired by recent studies in the high energy physics literature of Jackiw-Teitelboim (JT) quantum gravity and its holographic properties, see \cite{Stanford_JT_2020,Saad_JT_2019} and \cite{Mertens_Solvable_2023} for a survey.
In short, JT gravity is a two-dimensional toy model of gravity coupled to a dilaton field that ensures dynamically that the geometry is constantly curved.
Whereas most works in JT quantum gravity have focused on the degrees of freedom in the boundary geometry in an asymptotic (or non-compact) limit, it is desirable to understand the theory also at \emph{finite cutoff}, where the field theory should describe compact constant curvature metrics on surfaces with boundary.  
Ferrari recently introduced a lattice approach \cite{ferrari2025jackiw,ferrari2024randomdisksconstantcurvature} to this question, proposing that in Euclidean signature and zero curvature such a field theory on the disk may be obtained from a scaling limit of uniform random flat quadrangulation with a boundary of prescribed length.
The enumeration problem associated to the precise model put forward in \cite[Section~3.5]{ferrari2025jackiw} is still open, although several steps towards a matrix integral solution were described in \cite[Section~5]{ferrari2025jackiw}.
The rigid quadrangulations introduced in this work can be understood as a subclass of flat quadrangulations for which the enumeration problem is tractable.

\subsection{Flat and rigid quadrangulations}\label{sec:rigiddef}

To introduce our main combinatorial family we recall some terminology about planar maps.
A \emph{planar map} is a connected multigraph embedded in the plane, viewed up to orientation-preserving homeomorphisms of the sphere.
We will always take maps to be \emph{rooted}, meaning that a distinguished oriented edge is specified.
The vertex at the origin of the root and the face on the left of the root are called the \emph{root vertex} and the \emph{root face}, respectively.
A \emph{quadrangulation of the disk of perimeter }$2k$ is a rooted planar map with all faces of degree $4$ except for the root face of degree $2k$.
Such a map will always be drawn in the plane with the root face on the outside, which we call the \emph{boundary} and orient in counterclockwise direction, and we suffice with indicating only the root vertex (green dot).
Vertices and edges that are not on the boundary are said to be \emph{inner}.
A quadrangulation (of the disk) is \emph{flat} if each inner vertex has degree $4$ and its boundary is simple (meaning that the $2k$ boundary vertices are distinct).
A boundary vertex or corner is called \emph{convex} if it is of degree $2$, \emph{straight} if it is of degree $3$, or \emph{concave} if it is of degree $4$ or larger.
In a flat quadrangulation we require the root vertex to be convex.
A \emph{side} of a flat quadrangulation is a maximal sequence of boundary edges connected by straight boundary vertices.

\begin{figure}
    \centering
    \includegraphics[width=.9\linewidth]{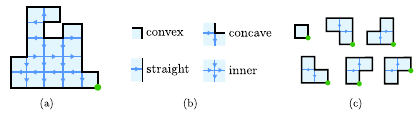}
    \caption{(a) Example of a rigid quadrangulation $\mathbf{r} \in \mathcal{R}_8$. (b) The four types of vertices that occur. (c)~Illustrating the $|\mathcal{R}_3| = 1$ and $|\mathcal{R}_4| = 5$ rigid quadrangulations with $3$ respectively $4$ non-root convex corners. \label{fig:rigidquadrangulation}}
\end{figure}

A \emph{ray} in a flat quadrangulation is a path of inner edges starting and ending on a boundary vertex but otherwise only visiting inner vertices, where it is required to go straight (i.e.\ consecutive edges meet an inner vertex at opposite sides).
Note that every inner edge belongs to a unique (unoriented) ray.
A \emph{rigid quadrangulation (of the disk)} is a flat quadrangulation in which all concave vertices are of degree $4$ and in which each ray connects a concave corner to a straight boundary vertex.
In this case, each ray and therefore each inner edge is naturally oriented from concave to straight.
See Figure~\ref{fig:rigidquadrangulation}a for an example.
These conditions for a quadrangulation of the disk with oriented inner edges to be rigid can equivalently be described in terms of allowed types of vertices shown in Figure~\ref{fig:rigidquadrangulation}b.
For $n\geq 3$, the set of rigid quadrangulations with $n+1$ convex corners and $n-3$ concave corners is denoted $\mathcal{R}_n$, and $\mathcal{R} = \bigcup_{n\geq 3} \mathcal{R}_n$.
All rigid quadrangulations for $n=3,4$ are shown in Figure~\ref{fig:rigidquadrangulation}c.
Our first enumerative result is the following.

\begin{theorem}\label{thm:rigidquadenum}
    The generating function $Z(t) = \sum_{n=3}^\infty |\mathcal{R}_n| t^n = t^3 + 5 t^4 + 33 t^5 + \cdots$ of rooted rigid quadrangulations of the disk is given by $Z(t) = \tfrac{1}{4}(t-2t^2 - R(t))$, where $R(t)=t-2t^2 - \cdots$ is the unique formal power series solution to
    \begin{equation}
        \sum_{n=0}^\infty \frac{1}{n+1} \binom{2n}{n}^2 R(t)^{n+1} = t.
    \end{equation}
    In particular, the number of rigid quadrangulations is asymptotic to
    \begin{equation}
        |\mathcal{R}_n| \sim \frac{(4\pi)^n}{16 \,n^2 \log^2 n} \quad\text{as } n\to\infty.
    \end{equation} 
\end{theorem}

\noindent
This sequence appears in the OEIS as \cite[A324312]{oeis}.

\subsection{Bijection with colorful quadrangulations}\label{sec:bijectionintro}

A \emph{$\mathbb{Z}$-labeled quadrangulation} is a (rooted) quadrangulation together with a $\mathbb{Z}$-labeling of its vertices, such that labels at the endpoints of each edge differ by exactly $1$, and such that the root edge points from label $1$ to label $0$.
Following \cite{Bousquet-Melou_generating_2020}, we let a \emph{colorful ($\mathbb{Z}$-labeled) quadrangulation} be a $\mathbb{Z}$-labeled quadrangulation in which for each face the labels read $(r-1,r,r+1,r)$ in cyclic order for some $r\in\Z$ (see Figure~\ref{fig:rigidandlabeledquad}b). 
This terminology stems from the observation that one can equivalently think of the labeling as a proper 3-coloring of the vertices, by considering the labels modulo $3$, and then the condition amounts to every face receiving all three colors.

\begin{figure}
    \centering
    \includegraphics[width=.6\linewidth]{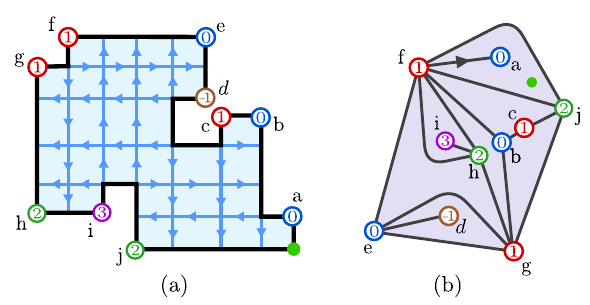}
    \caption{(a) The same rigid quadrangulation $\mathbf{r}$ as before with all non-root convex corners labeled by their turning number. (b) The associated colorful quadrangulation $\mathfrak{q}=\Psi(\mathbf{r})$. The identification $v(c)$ of its vertices with the convex corners of $\mathbf{r}$ is indicated by letters. \label{fig:rigidandlabeledquad}}
\end{figure}

To any rigid quadrangulation $\mathbf{r}$ one can associate a colorful quadrangulation $\mathfrak{q}=\Psi(\mathbf{r})$ as follows.
First we label each convex corner $c$ of $\mathbf{r}$ except the root corner by its \emph{turning number} $\tau(c)$, which is the number of convex corners minus the number of concave corners encountered in counterclockwise direction around the boundary strictly between the root corner and $c$ (see Figure~\ref{fig:rigidandlabeledquad}a).
Next we associate a topological sphere to $\mathbf{r}$ by considering its \emph{double}, in the sense that we glue the boundary of $\mathbf{r}$ to a mirror copy of itself that we think of as the backside of the disk.
We let the vertices of $\mathfrak{q}$ be all the convex corners of $\mathbf{r}$ except for the root corner, and we view them as marked points in the double, labeled by their turning number.
\begin{figure}[b]
    \centering
    \includegraphics[width=\linewidth]{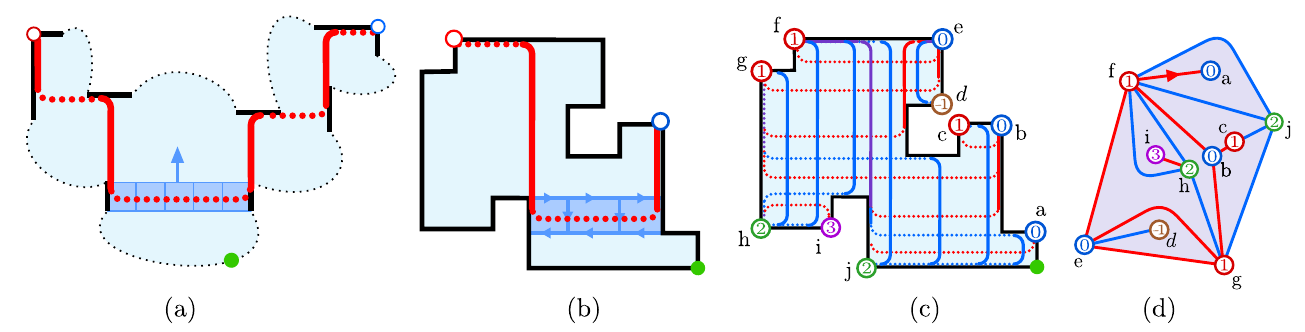}
    \caption{(a) Schematic illustration of the edge of $\mathfrak{q}$ associated to a row of $\mathbf{r}$. The solid vertical segments are drawn on the front and the dotted horizontal segments on the back of the double. (b)~Example edge for a single row. (c) All edges of $\mathfrak{q}$: those associated to the rows in red, and to the columns in blue. (d) The colorful quadrangulation $\mathfrak{q}$ with even edges in red and odd edges in blue.\label{fig:edgedrawing}}
\end{figure}
Then we draw an edge of $\mathfrak{q}$ in the double for each row and each column of $\mathbf{r}$ as follows.
A row or column has a natural transverse orientation pointing away from the root vertex (see the arrow in Figure~\ref{fig:edgedrawing}a).
Then the middle section of the row or column is extended to a path connecting two convex corners by alternating the transverse and parallel direction whenever hitting the side of the disk, as illustrated in Figure~\ref{fig:edgedrawing}a.
By convention the path travels vertically on the front of the double and horizontally on the back.
With this convention the edges of $\mathfrak{q}$ do not cross, but several edges ending on the same vertex may merge before reaching that vertex.
However, deforming the edge paths slightly to resolve the merging leads to a unique planar map $\mathfrak{q}$ in the double (Figure~\ref{fig:edgedrawing}c \& d).
It has a distinguished face with labels $(0,1,2,1)$ containing the root corner of $\mathbf{r}$.
We root $\mathfrak{q}$ on the unique edge pointing from $1$ to $0$ that has this distinguished face on its right.

For $n \geq 3$, let $\mathcal{C}_n$ denote the set of colorful quadrangulations with $n$ vertices such that the face to the right of the root has labels $(0,1,2,1)$.
Note that this is precisely half of all colorful quadrangulations with $n$ vertices, since the face to the right of the root can only have labels $(0,1,2,1)$ and $(0,1,0,-1)$ and there is a symmetry relating both sets.
A generating function for $\mathcal{C} = \bigcup_{n\geq 3} \mathcal{C}_n$ with control on the number of faces was obtained by Bousquet-M\'elou and Elvey Price in \cite[Theorem~1.1]{Bousquet-Melou_generating_2020}.
Theorem \ref{thm:rigidquadenum} is then a direct consequence of the following.

\begin{theorem}\label{thm:rootedbijection}
    This mapping $\Psi : \mathcal{R} \to \mathcal{C}$ is a bijection between rigid quadrangulations and colorful quadrangulations. Moreover, there is a one-to-one correspondence between the non-root convex corners of $\mathbf{r}$ and the vertices of $\mathfrak{q} = \Psi(\mathbf{r})$ such that the turning number of a corner agrees with the label of the corresponding vertex.
\end{theorem}

Notice that the edge drawing algorithm on the double of $\mathbf{r}$ is unchanged when $\mathbf{r}$ is replaced by its mirror image $\mathsf{Mirror}(\mathbf{r})$ (which amounts to flipping over the double such the front of $\mathbf{r}$ becomes the back of $\mathsf{Mirror}(\mathbf{r})$). 
The only change is the labeling of the convex corners by their turning number. 
Denoting by $\mathsf{Relabel}(\mathfrak{q})$ the colorful quadrangulation obtained from $\mathfrak{q}$ by relabeling all vertices by $j \mapsto 2-j$, and appropriately moving the root edge along the face on its right to make it point from $1$ to $0$ again, we have the relation 
\begin{align}
    \Psi \circ \mathsf{Mirror} = \mathsf{Relabel} \circ \Psi.\label{eq:Psisymmetry}
\end{align}
This will come in handy when relating statistics between rigid and colorful quadrangulations.

\subsection{Rectilinear disks}\label{sec:rectilinear}

The definition of rigid quadrangulations may seem a bit artificial, but we will argue now that it arises naturally in the context of polygonal disks with real side lengths.
Let $\mathbb{D}$ be the unit disk in $\mathbb{R}^2$.
We say that a Riemannian metric $g$ on $\mathbb{D}$ is \emph{rectilinear} if $g$ is flat in the interior of $\mathbb{D}$, and the boundary $\partial\mathbb{D}$ is piece-wise geodesic with corners of interior angle $\pi/2$ or $3\pi/2$, called \emph{convex corners} and \emph{concave corners}, respectively.
A \emph{rooted} rectilinear metric is such a metric together with a distinguished convex corner.
Then a \emph{rectilinear disk} $D$ is an equivalence class of rooted rectilinear metrics under isometries that preserve the orientation and the root.
For an integer $n\geq 3$ we let $\mathcal{M}_{n}$ be the moduli space of rectilinear disks with $2n-2$ corners.

A rectilinear disk $D\in\mathcal{M}_{n}$ admits a locally isometric immersion $f_D : D \to \mathbb{R}^2$ into the Euclidean plane $\mathbb{R}^2$, which becomes unique when we require that the root is sent to the origin with the disk aligned with the top-left quadrant (see Figure~\ref{fig:rectilinear}a).
It can happen that $f_D$ is injective, so that $D$ is isometric to a (rectilinear) Jordan domain in the Euclidean plane, but generally $f_D$ is not injective, i.e.\ its image may self-overlap in $\mathbb{R}^2$. 
However, one may see that the rectilinear disks in a neighborhood of $D$ in $\mathcal{M}_{n}$ are parametrized by $2n-4$ side lengths $L_1,\ldots,L_{2n-4}$, namely all but one of the horizontal and all but one of the vertical side lengths (the remaining two side lengths are determined by closure).
We may thus equip $\mathcal{M}_{n}$ with a properly normalized $(2n-4)$-dimensional Lebesgue measure
\begin{align}
    \mu_n \coloneqq (2n-5)!\,\rmd L_1\rmd L_2 \cdots \rmd L_{2n-4}.
\end{align}
Let $\mathsf{Perim} : \mathcal{M}_{n} \to \R_{>0}$ be the half-perimeter of the disk.
By scaling symmetry the push-forward measure $\mathsf{Perim}_* \mu_n = V_n(L)\rmd L$ on $\R$ has a density $V_n(L)$ that is a monomial of degree $2n-5$, expressing the volume of the moduli space of rectilinear disks with fixed half-perimeter $L$.
For example, the rectilinear disks $D \in \mathcal{M}_3$ are rectangles of side lengths $L_1$, $L_2$ and therefore $\mathsf{Perim}(D) = L_1 + L_2$.
We thus find $V_3(L) = L$, because $\mathsf{Perim}_* \mu_3 = \mathsf{Perim}_*(\rmd L_1\rmd L_2) = L \rmd L$.

\begin{figure}
    \centering
    \includegraphics[width=.7\linewidth]{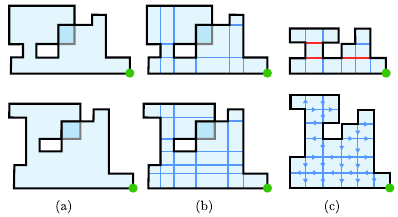}
    \caption{(a) Two examples of rectilinear disks in $\mathcal{M}_{10}$. (b) Their subdivisions into rectangles. (c) The corresponding combinatorial type. Both are flat quadrangulations, but only the bottom one is a rigid quadrangulation in the sense that every ray starts at a concave corner and ends on a straight boundary vertex. This is not the case for the rays marked in red in the top example, which start and end on concave vertices. \label{fig:rectilinear}}
\end{figure}

\begin{figure}[h]
    \centering
    \includegraphics[width=.9\linewidth]{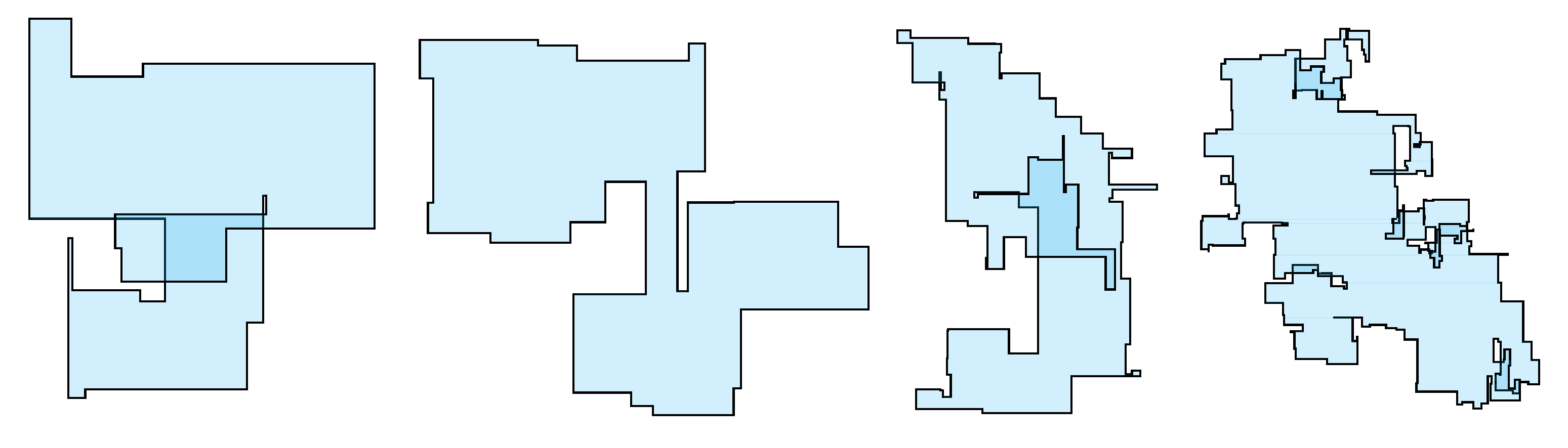}
    \caption{Uniform random rectilinear disks of fixed perimeter immersed in the plane.\label{fig:uniformrect}}
\end{figure}

To a rectilinear disk $D$ we may naturally associate a rooted flat quadrangulation $\mathsf{Type}(D)$, that we call the \emph{combinatorial type} of $D$.
It is obtained by drawing for each of the $n-3$ concave corners two rays extending the adjacent sides inwards until hitting an opposite side (Figure~\ref{fig:rectilinear}b).
In this way we obtain a subdivision of $D$ into rectangles and a corresponding flat quadrangulation $\mathbf{f} = \mathsf{Type}(D)$ rooted at a convex corner.
Since pairs of rays may overlap, $\mathbf{f}$ has at most $n-2$ rows and at most $n-2$ columns, say $k \leq 2n-4$ in total.
It is easily seen that the family $\mathsf{Type}^{-1}(\mathbf{f}) \subset \mathcal{M}_n$ of disks sharing the same combinatorial type $\mathbf{f}$ is parametrized precisely by the $k$ widths $w_1,\ldots,w_k > 0$ of the rows and columns of $\mathbf{f}$.
This gives rise to a cell decomposition of $\mathcal{M}_n$, in which the top-dimensional cells are indexed by flat quadrangulations with the maximal number $k=2n-4$ of rows and columns.
These are precisely the rigid quadrangulations $\mathcal{R}$.
Expressed in terms of the $2n-4$ widths, the measure $\mu_n$ on $\mathsf{Type}^{-1}(\mathbf{r}) \subset \mathcal{M}_n$ for $\mathbf{r}\in\mathcal{R}$ is
\begin{align*}
    \mu_n = (2n-5)!\,\rmd w_1\rmd w_2 \cdots \rmd w_{2n-4},
\end{align*}
while $\mathsf{Perim}(D) = \sum_{i=1}^{2n-4} w_i$.
Hence, the volume of the rectilinear disks of half-perimeter $L$ in $\mathsf{Type}^{-1}(\mathbf{r})$ is $L^{2n-5}$ independent of $\mathbf{r}$.
We conclude the following.

\begin{corollary}
    The volume $V_n(L)$ of the moduli space $\mathcal{M}_n$ of rectilinear disks of half-perimeter $L$ and $n+1$ convex corners is $V_n(L) = |\mathcal{R}_n|\, L^{2n-5}$, with $|\mathcal{R}_n|$ as in Theorem~\ref{thm:rigidquadenum}.
\end{corollary}

One may also interpret this result probabilistically in the following sense.
For any $L > 0$, upon conditioning on $\mathsf{Perim}(D)=L$, the measure $\mu_n$ defines the \emph{uniform random rectilinear disks} of half-perimeter $L$ (see Figure~\ref{fig:uniformrect}).
Since its combinatorial type is almost surely a rigid quadrangulation, its law is that of the uniform random rigid quadrangulation in $\mathcal{R}_n$ (see Figure~\ref{fig:uniformrigid} on the front page) with the sequence of widths of its rows and columns given by an independent uniform point in the $(2n-5)$-dimensional simplex $\{\sum_{i=1}^{2n-4} w_i = L\}$.

\subsection{Dictionary}\label{sec:dictionary}

The bijection $\Psi$ relates a variety of natural statistics on rigid quadrangulations to equally natural statistics on colorful quadrangulations.
In order to discuss this dictionary, it is useful to introduce the notion of \emph{geodesic to the root} from a point in a rigid quadrangulation $\mathbf{r}$.
It is most conveniently defined by considering an arbitrary rectilinear disk $D$ with combinatorial type $\mathsf{Type}(D) = \mathbf{r}$, in the sense of Section~\ref{sec:rectilinear}.
Then for a point $x\in D$ we may consider the shortest curve $\Gamma_x$ in $D$ to the root corner with respect to its (Riemannian) metric structure.  
The curve $\Gamma_x$ is piecewise linear with bends only at concave corners of $D$.
We say a point $x$ on the boundary $\partial D$ is \emph{left-tangential} (respectively \emph{right-tangential}) if $\Gamma_x$ starts tangentially to the boundary towards the left (respectively right), as seen from the interior of $D$.
If $x$ is a (non-root) convex corner of $D$ then we qualify the corresponding convex corner of $\mathbf{r}$ accordingly.
If $x$ is on a side of $\partial D$ and it is left/right-tangential then all points on that side are seen to be left/right-tangential, and we say that the corresponding side of $\mathbf{r}$ (i.e.\ maximal sequence of boundary edges connected by straight boundary vertices) is left/right-tangential. 
Observe that a left/right-tangential corner is adjacent to exactly one left/right-tangential side.
We let the \emph{degree} of such a corner be the length of the adjacent side (in the rigid quadrangulation).
It is not hard to see that all these properties are independent of the choice of rectilinear representative $D \in \mathsf{Type}^{-1}(\mathbf{r})$.

\begin{figure}[h]
    \centering
    \includegraphics[width=.7\linewidth]{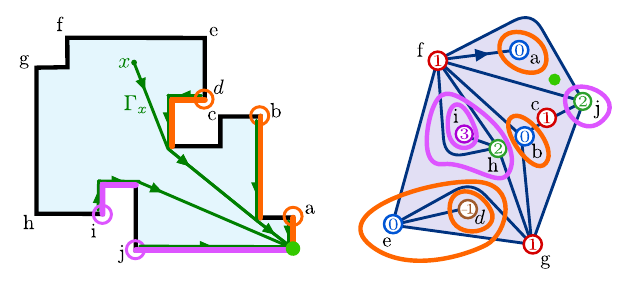}
    \caption{
    A rectilinear representative of the rigid quadrangulation of Figure~\ref{fig:rigidandlabeledquad}. The shortest curve $\Gamma_x$ to the root for a point $x$ is shown in green. The 4 right-tangential sides and 3 right-tangential convex corners are indicated in orange, and the 3 left-tangential sides and 2 left-tangential convex corners in pink. The corresponding decreasing (orange) and increasing (pink) level lines on the colorful quadrangulation are indicated on the right. Note that a local minimum is circled by a decreasing level line and a local maximum by an increasing level line.
    \label{fig:minmaxlevel}}
\end{figure}

In a colorful quadrangulation $\mathfrak{q}$, a vertex is called a \emph{local maximum} if all its neighbors have smaller label or \emph{local minimum} if all its neighbors have larger label.
A \emph{level line} on $\mathfrak{q}$ is a simple closed curve on its dual that crosses only edges with identical label pair.
Note that the union of all level lines of $\mathfrak{q}$ can be drawn in a non-intersecting fashion.
A level line is said to be \emph{increasing} (respectively \emph{decreasing}) if the root is contained on the side of the level line with lower (respectively higher) label.

\begin{theorem}\label{thm:dictionary}
    Let $\mathbf{r}\in\mathcal{R}$ and $\mathfrak{q} = \Psi(\mathbf{r})$ be drawn on the double of $\mathbf{r}$. Then the following list of correspondences applies to the pair $(\mathbf{r},\mathfrak{q})$:
    \begin{center}
    \newcounter{tableitem}
    \renewcommand{\thetableitem}{(\roman{tableitem})}
    \newcommand{\reftableitem}{\refstepcounter{tableitem}\thetableitem} 
    \begin{tabular}{r|l|l}
        & Rigid quadrangulation $\mathbf{r}$ & Colorful quadrangulation $\mathfrak{q}$ \\\hline
        \reftableitem\label{it:vertex} & non-root convex corner & vertex \\
        \reftableitem\label{it:evenedge} & row & even edge \\
        \reftableitem\label{it:oddedge} & column & odd edge \\
        \reftableitem\label{it:face} & root-corner or concave corner & face\\
        \reftableitem\label{it:rightconvex} & right-tangential convex corner of degree $\ell$ & local minimum of degree $\ell$ \\
        \reftableitem\label{it:leftconvex} & left-tangential convex corner of degree $\ell$ & local maximum of degree $\ell$ \\
        \reftableitem\label{it:rightside} & right-tangential side& decreasing level line \\
        \reftableitem\label{it:leftside} & left-tangential side & increasing level line\\\hline
    \end{tabular}
\end{center}
\end{theorem}

\noindent
This dictionary can be further refined. 
For instance, the labels of vertices, edges, faces and level lines of $\mathfrak{q}$ can be deduced rather easily from the turning numbers in the vicinity of the corresponding elements in the rigid quadrangulation.
With regard to the structure of level lines, it is convenient to consider the \emph{extension} of a tangential side: if its endpoint furthest from the root corner is concave we take its extension to be the side itself together with the parallel ray from this concave corner, or just the side itself when the endpoint is convex.
The length of a level line then agrees with the length of the extension of the corresponding tangential side in the rigid quadrangulation.
This extension also allows to see the nesting structure of level lines in $\mathfrak{q}$: level line A separates level line B from the (center of the) face to the right of the root in $\mathfrak{q}$ if and only if the extension of the tangential side corresponding to A in $\mathbf{r}$ separates the tangential side corresponding to B from the root corner.

\subsection{Refined bijections}

The bijection $\Psi$ can be generalized to the setting of colorful quadrangulations with a boundary.
To be precise, a \emph{colorful quadrangulations of the disk} is a labeled quadrangulations of the disk in which only the inner faces are required to satisfy the colorful condition, while the not necessarily simple boundary may be arbitrarily labeled except that the root is still required to point from label $1$ to $0$.

Let us consider the partition $\mathcal{R} = \bigcup_{k\geq 1} \bigcup_{\ell_1,\ldots,\ell_k \geq 1} \mathcal{R}^{(\ell_1,\ldots,\ell_k)}$ where $\mathcal{R}^{(\ell_1,\ldots,\ell_k)}$ is the set of rigid quadrangulations in which the boundary between the root corner and the next convex corner encountered in clockwise direction consists of $k$ sides of length $\ell_1,\ldots,\ell_k$ separated by $k-1$ concave corners (see Figure~\ref{fig:boundarybijection} for an example with $k=3$).
We call this portion of the boundary of length $\ell\coloneqq \ell_1+\cdots+\ell_k$ the \emph{$k$-fold base} of the rigid quadrangulation, or simply the \emph{base} if the context is clear.
Let $\Psi^{\mathrm{b}}(\mathbf{r})$ be the rooted labeled map $\Psi(\mathbf{r})$ from which all edges are deleted that correspond to columns or rows that end on the base of $\mathbf{r}$ as well as the vertex corresponding to the endpoint of the base.
All deleted edges were adjacent to a vertex of label $2$ or higher, so the root edge is not affected.
Let further $\mathsf{Walk}^{(\ell_1,\ldots,\ell_k)} \in \Z^{2\ell}$ be the walk on the integers defined as follows. 
The walk starts with $\ell_1$ iterations of $(0,1)$, followed by $\ell_3$ iterations of $(2,3)$ and so on, meaning that it has $\ell_i$ iterations of $(i-1,i)$ for all odd indices $i \leq k$.  
This is followed by $\ell_i$ iterations of $(i,i-1)$ for all even indices $i \leq k$ but now in descending order.
Note that $\mathsf{Walk}^{(\ell_1,\ldots,\ell_k)}$ is a walk with $\pm1$ steps starting at $0$ and ending at $1$.
Moreover, the corresponding walk on $\Z+\tfrac12$ given by the midpoints of the steps of $\mathsf{Walk}^{(\ell_1,\ldots,\ell_k)}$ is \emph{unimodal}, meaning that it increases weakly up to a peak and then decreases weakly.
Denote by $\mathcal{C}^{(\ell_1,\ldots,\ell_k)}$ the set of colorful quadrangulations of the disk of perimeter $2(\ell_1+\cdots+\ell_k)$ satisfying the \emph{unimodal boundary condition} that the labels in clockwise order around the boundary are given by $\mathsf{Walk}^{(\ell_1,\ldots,\ell_k)}$, starting and ending at the endpoints of the root.

\begin{theorem}\label{thm:basebijection}
 For each $k \geq 1$ and $\ell_1,\ldots,\ell_k \geq 1$, the mapping $\Psi^{\mathrm{b}}$ restricts to a bijection
 \begin{align*}
 \mathcal{R}^{(\ell_1,\ldots,\ell_k)} \to \mathcal{C}^{(\ell_1,\ldots,\ell_k)}.
 \end{align*}
\end{theorem}

\begin{figure}[t]
    \centering
    \includegraphics[width=.7\linewidth]{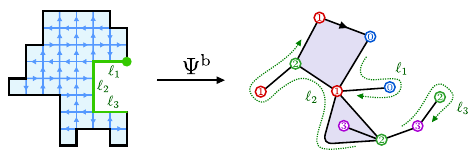}
    \caption{Example of a rigid quadrangulation $\mathbf{r} \in \mathcal{R}^{(2,3,2)}$ and the corresponding colorful quadrangulation of the disk $\Psi^{\mathrm{b}}(\mathbf{r})\in \mathcal{C}^{(2,3,2)}$. The labels in clockwise order around the boundary are $\mathsf{Walk}^{(2,3,2)} = ({\protect\underbrace{0,1,0,1}_{\ell_1}},\protect\underbrace{2,3,2,3}_{\ell_3},\protect\underbrace{2,1,2,1,2,1}_{\ell_2})$.\label{fig:boundarybijection}}
\end{figure}

Most of this work is devoted to proving this theorem for $k=1$. 
Then Theorem~\ref{thm:rootedbijection} will be interpreted as a special case of this theorem for $k=1$ and $\ell_1=1$. (Actually it is more natural to think of Theorem~\ref{thm:rootedbijection} as the special case $k=2$ and $\ell_1=\ell_2=1$, since $\mathcal{C}^{(1,1)} = \mathcal{C}$ by definition. However, that does not play well with the logic of the proof.)
As explained in Section~\ref{sec:zipping}, this follows from the fact that one can ``expand'' a rigid quadrangulation into a base-1 rigid quadrangulations and one can ``zip'' the boundary of perimeter-2 colorful quadrangulation.
Finally, Theorem~\ref{thm:basebijection} for arbitrary $k\geq 1$ will be deduced from Theorem~\ref{thm:rootedbijection} by inspection.

We mention here two further natural specializations of $\Psi^{\mathrm{b}}$.
For $p,q \ge 1$, we denote by $\mathcal{R}_{\mathrm{B}}^{(p)(q)} \subset \mathcal{R}^{(p)} \subset \mathcal{R}$ the \emph{B-type} rigid quadrangulations for which both sides adjacent to the root end in convex corners and are of length $p$ (horizontal side) and $q$ (vertical side) respectively.
We can further restrict to the \emph{C-type} rigid quadrangulations $\mathcal{R}_{\mathrm{C}}^{(p)(q)} \subset \mathcal{R}_{\mathrm{B}}^{(p)(q)}$ that fully contain a rectangle of size $p\times q$ with a corner at the root. 
On the other hand, let $\mathcal{C}^{(p)(q)}_{\mathrm{B}} \subset \mathcal{C}^{(p)}$ be the \emph{B-type} colorful quadrangulations of the disk of perimeter $2p$ with labels $0,1,0,1,\ldots$ for which the endpoint of the root edge with label $0$, denoted $v_0$, has degree $q$ and $v_0$ is required to only have neighbors of label $1$.
Moreover, the \emph{C-type} colorful quadrangulations $\mathcal{C}^{(p)(q)}_{\mathrm{C}} \subset \mathcal{C}^{(p)(q)}_{\mathrm{B}}$ are defined by further restricting $v_0$ to be only incident once to the root face.
\begin{figure}[t]
    \centering
    \includegraphics[width=.6\linewidth]{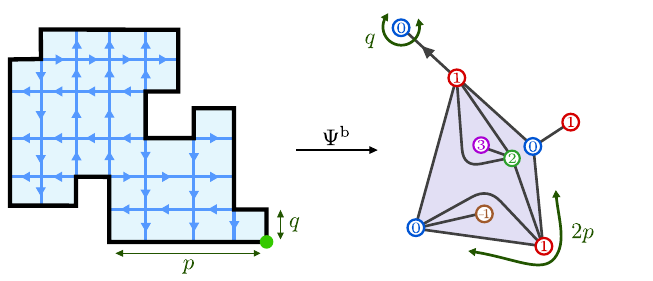}
    \caption{Example of a B-type rigid quadrangulation $\mathbf{r}\in \mathcal{R}_{\mathrm{B}}^{(p)(q)}$ for $p=4$, $q=1$ and its image under $\Psi^{\mathrm{b}}$. In this case $\mathbf{r}\in \mathcal{R}_{\mathrm{C}}^{(p)(q)}$ if of type C as well, because it contains a rectangle of size $4\times 1$ (corresponding to the bottom row).
    \label{fig:doublebase}}
\end{figure}
\begin{corollary}\label{cor:doublebase}
    For each $p,q \geq 1$, the mapping $\Psi^{\mathrm{b}}$ restricts to bijections
    \begin{align*}
        \mathcal{R}_{\mathrm{B}}^{(p)(q)} \to \mathcal{C}_{\mathrm{B}}^{(p)(q)}\qquad \text{and} \qquad \mathcal{R}_{\mathrm{C}}^{(p)(q)} \to \mathcal{C}_{\mathrm{C}}^{(p)(q)}.
    \end{align*}
\end{corollary}

\subsection{Multivariate generating functions}
Up to minor modifications, the colorful quadrangulations of the disk $\mathcal{C}^{(p)}$, the B-type colorful quadrangulations $\mathcal{C}_{\mathrm{B}}^{(p)(q)}$ and the C-type colorful quadrangulations $\mathcal{C}_{\mathrm{C}}^{(p)(q)}$ are precisely the \emph{P-patches}, \emph{B-patches} and \emph{C-patches} of \cite[Section~6]{Bousquet-Melou_generating_2020}, while $\mathcal{C}^{(p,q)}$ corresponds to the \emph{E-patches} of \cite[Definition~2.8]{Bousquet-Melou_Refined_2025} restricted to the colorful case ($v=1$ and $\omega=0$ in their notation).
The only differences are that the orientation of the root edge is reversed in our definitions and that we do not consider the cases $p=0$ or $q=0$.
Combining Theorem~\ref{thm:basebijection} and Corollary~\ref{cor:doublebase} with the explicit generating functions for colorful quadrangulations obtained by Bousquet-M\'elou and Elvey Price in \cite[Theorem~7.1]{Bousquet-Melou_generating_2020} leads to the following.
\begin{corollary}\label{cor:multivariate}
    Writing $n(\mathbf{r})$ for the number of non-root convex corners of a rigid quadrangulation $\mathbf{r}$, the formal multivariate generating functions of rigid quadrangulations in $\mathcal{R}^{(p)}$, $\mathcal{R}^{(p)(q)}_{\mathrm{B}}$, $\mathcal{R}^{(p)(q)}_{\mathrm{C}}$ and $\mathcal{R}^{(p,q)}$ are given by 
    \begin{align*}
        P(t,y) &= \sum_{p \geq 1} y^p \,\,\,\sum_{\mathbf{r}\in\mathcal{R}^{(p)}} t^{n(\mathbf{r})-1} \,\,\,\,\,\,\,\,\,= \sum_{p\geq 1} y^p \sum_{n\geq p} \frac{1}{n+1}\binom{2n}{n}\binom{2n-p}{n} R(t)^{n+1},\\
        B(t,x,y) &= \sum_{p,q \geq 1} y^{p} x^{q} \!\sum_{\mathbf{r}\in\mathcal{R}^{(p)(q)}_{\mathrm{B}}} t^{n(\mathbf{r})-2} = \exp\left( \sum_{n\geq 0} \sum_{0\leq i,j\leq n} \frac{1}{n+1}\binom{2n-i}{n}\binom{2n-j}{n} x^{i+1}y^{j+1} R(t)^{n+1}\right)-1,\\
        C(t,x,y) &= \sum_{p,q \geq 1} y^{p} x^{q} \sum_{\mathbf{r}\in\mathcal{R}^{(p)(q)}_{\mathrm{C}}} t^{n(\mathbf{r})-2} = 1- \frac{1}{1+B(t,x,y)},\\
        E(t,x,y) &= \sum_{p,q \geq 1} y^{p} x^{q} \sum_{\mathbf{r}\in\mathcal{R}^{(p,q)}} t^{n(\mathbf{r})-2} \,\,\,= \frac{1}{x y} C(t,x,y) - P(t,x)-P(t,y) - t,
    \end{align*}
    where $R(t)$ is the same power series as in Theorem~\ref{thm:rigidquadenum}.
\end{corollary}
\noindent
The first few terms in these power series are 
\begin{align*}
P(t,x)& = (t^2 + 2t^3 + 10t^4 + \cdots)\,x + (2t^3 + 8 t^4 + 50 t^5 \cdots)\,x^2 + \cdots\\
B(t,x,y) &= t x y + (t^2 + 2t^3 + 10 t^4 + \cdots)(x y^2+x^2 y) + (t^2 +t^3 + 5 t^4 + \cdots)\,x^2y^2+\cdots\\
C(t,x,y) &= t x y + (t^2 + 2t^3 + 10 t^4+ \cdots)(x y^2+x^2 y) + (t^3 + 5 t^4 + \cdots)\,x^2y^2+\cdots\\
E(t,x,y) &= (t^3 + 5t^4 + 33 t^5 + \cdots)\,x y + (2t^4 + 15 t^5 + \cdots)(xy^2+x^2 y) + \cdots.
\end{align*}

Bousquet-M\'elou and Elvey Price established the enumeration in \cite[Theorem~7.1]{Bousquet-Melou_Eulerian_2020} by considering several decompositions of colorful quadrangulations that give rise to a system of equations for the generating functions of P-, B-, and C-patches.
It is a 2-catalytic system of equations, in the sense that it involves two additional generating variables $x$ and $y$ besides $t$, that they show admits unique solutions.  
We will demonstrate in Section~\ref{sec:catalyticexplanation} their 2-catalytic decompositions of P-, B-, and C-patches have natural counterparts at the level of rigid quadrangulations.
In a very recent work, Bousquet-M\'elou and Elvey Price obtained also a 1-catalytic equation for P-patches \cite[Proposition~9.3]{Bousquet-Melou_Refined_2025}.
This one happens to correspond exactly to the Tutte-like decomposition at the first step in the peeling exploration of the colorful quadrangulation (Section~\ref{sec:peeling}), or equivalently the first step in the row-by-row exploration of the rigid quadrangulation (Section~\ref{sec:rowbyrow}). 

Let us also highlight the concurrent work of Zonneveld \cite{Zonneveld_tree_2025} that provides an alternative bijective proof of the generating functions in Theorem~\ref{thm:rigidquadenum} and  Corollary~\ref{cor:multivariate} via a bijection with certain labeled trees called \emph{H-trees}.
These H-trees arise from a variant of the row-by-row exploration described in Section~\ref{sec:rowbyrow}, tracking the number of downward rays seen throughout the exploration.
One may interpret the composition of Zonneveld's bijection with the bijection of Theorem~\ref{thm:basebijection} as a tree bijection for colorful quadrangulations, thus providing a bijective proof for the enumeration results of \cite[Theorem~7.1]{Bousquet-Melou_Eulerian_2020}.  

\subsection{Specialization: rigid fighting fish}

A related model of self-overlapping polygons in the square lattice that has received considerable attention in the combinatorial literature is that of \emph{fighting fish} \cite{Duchi_Fighting_2017,Duchi_Fighting_2017a}.
To introduce this family in our language, let us consider a rooted flat quadrangulation $\mathbf{f}$ of the disk as defined in Section~\ref{sec:rigiddef}.
The notion of turning number $\tau(c)$ of a non-root convex corner $c$ can be generalized from the rigid case to general flat quadrangulations by setting $\tau(c) = \sum_v 3-\deg(v)$, where the sum is over all boundary vertices $v$ encountered in counterclockwise order strictly between the root and $c$. 
Then one may check that $\mathbf{f}$ is a fighting fish in the sense of \cite{Duchi_Fighting_2017} if and only if all non-root convex corners have turning number $0$, $1$ or $2$.
Let $\mathcal{FF}_n$ be the set of fighting fish of perimeter $2n+2$, and $\mathcal{FF} = \bigcup_{n \geq 1} \mathcal{FF}_n$.
It was shown in \cite[Theorem~1]{Duchi_Fighting_2017} that
\begin{align}
    |\mathcal{FF}_n| = \frac{2}{(n+1)(2n+1)}\binom{3n}{n},
\end{align}
see OEIS sequence \cite[A000139]{oeis}.
Bijective proofs of this enumeration were obtained via bijections with two-stack sortable permutations \cite{Fang_Fighting_2018}, Tamari intervals \cite{Duchi_bijection_2023} and non-separable rooted planar maps \cite{Duchi_bijection_2022}.

\begin{figure}[h]
    \centering
    \includegraphics[width=.4\linewidth]{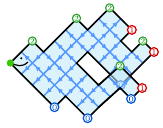}
    \caption{An example of a rigid fighting fish.\label{fig:rigidfightingfish}}
\end{figure}

In the context of our work it is then natural to consider \emph{rigid fighting fish} which are the fighting fish that are also rigid quadrangulation, and we denote the set of these by $\mathcal{RFF} = \mathcal{FF}\cap\mathcal{R}$.
According to Theorem~\ref{thm:rootedbijection}, $\Psi$ restricts to a bijection between $\mathcal{RFF}$ and colorful quadrangulations in which all vertices have label $0$, $1$ or $2$.
This bijection can be naturally composed with Tutte's bijection: drawing the diagonal that connects the vertices with label $1$ in each face and deleting all original edges (and properly transferring the root), we obtain a bijection with rooted Eulerian planar maps (i.e.\ rooted planar maps with all vertices of even degree).
The enumeration of these goes back to Tutte \cite{Tutte_census_1963}, see OEIS sequence \cite[A000257]{oeis}.

\begin{corollary}
    Composing $\Psi|_{\mathcal{RFF}}$ with Tutte's bijection gives a bijection between rigid fighting fish of perimeter $4 n$ (or equivalently $n+3$ convex corners) and rooted Eulerian planar maps with $n$ edges. In particular,
    \begin{align*}
        |\mathcal{RFF}_{2n-1}| = \frac{3\cdot2^{n-1}}{(n+1)(n+2)}\binom{2n}{n}.
    \end{align*}
\end{corollary}

It is natural to ask whether this bijection can be understood as a specialization of the bijection of Duchi and Henriet between fighting fishing and non-separable rooted planar Eulerian maps described in \cite{Duchi_bijection_2022}.

\subsection{Large random rigid quadrangulations and flat JT gravity}\label{sec:jt}

This work was inspired by questions of Ferrari in the context of JT gravity \cite{ferrari2024randomdisksconstantcurvature,ferrari2025jackiw} and answers several of them albeit in a slightly tweaked version of the model proposed in \cite{ferrari2025jackiw}.
Let us explain the relation and make explicit which answer can be extracted from this work.
In the language of \cite{ferrari2025jackiw} a \emph{self-overlapping polygon} (SOP) is a closed piecewise-linear path $\gamma$ in the plane that bounds a ``distorted'' disk, in the sense that $\gamma$ can be realized as the image of the boundary of a disk under an (orientation-preserving, topological) immersion.
If $\gamma$ is a simple closed curve, this immersion is unique up to precomposition by a homeomorphism.
But in general, a SOP can be realized as the boundary of a finite number of topologically distinct immersions.  
This number is called the \emph{multiplicity} $\mu_\gamma$ of $\gamma$.
We refer the reader to \cite{Poenaru_Extension_1995,Graver_When_2011,Evans_Combinatorial_2023} for background on the problem of finding such immersions.

Ferrari's suggestion is to consider the combinatorial family of SOPs on the square lattice viewed up to translation and counted with multiplicity. 
The corresponding generating function $W(t,g) = \sum_\gamma \mu_\gamma t^{F(\gamma)} g^{v(\gamma)}$ in \cite[Eq.~(3.12)]{ferrari2025jackiw} controls the length $v(\gamma)$ of $\gamma$ and the area $F(\gamma)$ of the disk as measured via pull back along an immersion (which is independent of the chosen immersion).
Since each immersion of such a SOP $\gamma$ defines a quadrangulation of the disk by pulling back the square lattice, the multiplicity $\mu_\gamma$ is precisely the number of pairs consisting of an (unrooted) quadrangulations $\mathbf{f}$ of the disk and an immersion of the quadrangulations that sends the boundary to $\gamma$.
For this to be a topological immersion $\mathbf{f}$ must be flat and have all boundary vertices of degree $2$, $3$ or $4$.
Since the immersion is unique up to rotation, $W(t,g)$ is also the generating function of flat quadrangulations together with, say, a choice of north direction, or the generating function of flat quadrangulations rooted at a convex corner and counted with the reciprocal of its number $n(\mathbf{f})$ of convex corners,
\begin{align}
    W(t,g) = \sum_{\substack{\text{rooted flat}\\\text{quadrangulations }\mathbf{f}}} \frac{1}{n(\mathbf{f})} t^{\mathsf{Faces}(\mathbf{f})} g^{\mathsf{Perimeter}(\mathbf{f})}.
\end{align} 
Ignoring the control on area by setting $t=1$, \cite[Section~4.3.1]{ferrari2025jackiw} conjectures the asymptotics of the enumeration $w_{\ell} \coloneqq [g^{\ell}]W(1,g)$ of SOPs of even length $\ell$ counted with multiplicity to be 
\begin{align}\label{eq:ferrariasymp}
    w_{\ell} \underset{\ell\to\infty}{\sim} a\, g_*^{-\ell} \ell^{-3} (\log \ell)^\varphi\qquad(\ell \in 2\Z)
\end{align}
for real constants $a$, $g_*$ and $\varphi$.
If we impose rigidity on the flat quadrangulations $\mathbf{f}$, and thus consider the enumeration $w^{\mathrm{rig}}_{\ell} \leq w_{\ell}$ of SOPs that bound a rigid quadrangulation, then $w^{\mathrm{rig}}_{\ell} = \tfrac{4}{\ell+8}[t^{(\ell+8)/4}] Z(t)$ for $\ell \in 4\Z$ and Theorem~\ref{thm:rigidquadenum} gives
\begin{align}
    w_{\ell}^{\text{rig}} \underset{\ell\to\infty}{\sim} (8\pi)^2 \left(4\pi\right)^{\ell/4} \ell^{-3} (\log n)^{-2}. \qquad(\ell \in 4\Z)
\end{align}
This is precisely of the form \eqref{eq:ferrariasymp}, suggesting that uniform rigid quadrangulations may be a good alternative to Ferrari's combinatorial model.
Amusingly, Ferrari remarks (just above \cite[Conjecture~4.4]{ferrari2025jackiw}) in the context of \eqref{eq:ferrariasymp} that ``pure JT gravity seems to share similarities with Liouville theory coupled to $c=1$ matter''.
Since colorful quadrangulations are expected to belong to the latter universality class (see Question~\ref{q:scalinglimit} below), the bijection $\Psi:\mathcal{R}\to\mathcal{C}$ suggests this similarity may be due to a precise correspondence between the two field theories.

To corroborate the potential relevance of large uniform random rigid quadrangulations to JT gravity at finite cutoff, we consider a flat analogue of the JT gravity trumpet that features prominently in the physics literature (see e.g.\ \cite{Stanford_JT_2020,Saad_JT_2019,Mertens_Solvable_2023}).
As illustrated in Figure~\ref{fig:rigidtrumpet}, the JT trumpet amplitude 
\begin{align}\label{eq:JTtrumpet}
    Z^{\mathrm{trumpet}}(\beta,b) = \frac{1}{\sqrt{2\pi \beta}} e^{-b^2 / (2\beta)}
\end{align}
should very roughly be interpreted as the partition function of hyperbolic metrics on the cylinder with a geodesic boundary of length $b$ and a moderately wiggling boundary of regularized length $\beta$ in an appropriate ``Schwarzian'' limit.

\begin{figure}[h]
    \centering
    \includegraphics[width=.87\linewidth]{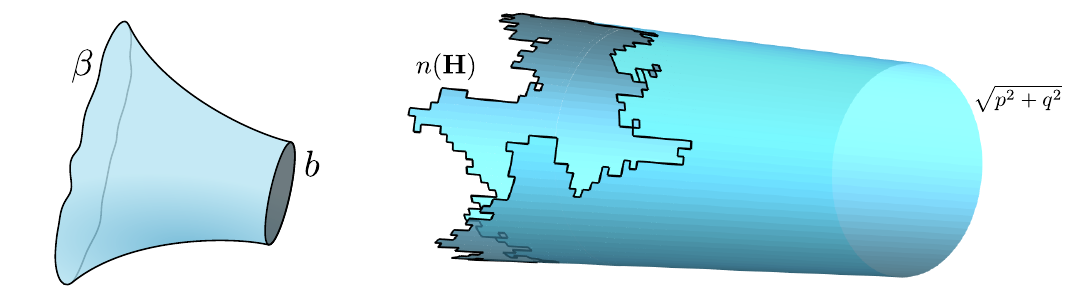}
    \caption{On the left an illustration of the JT trumpet as a hyperbolic cylinder with a geodesic boundary of length $b$ and an asymptotic ``wiggling'' boundary of parameter $\beta$. On the right an example of a rigid half-cylinder, which is actually a simulation of the Boltzmann $(p,q)$-rigid half-cylinder $\mathbf{H}$ with slightly subcritical parameter $t < 1/(4\pi)$.\label{fig:rigidtrumpet}}
\end{figure}

A flat analogue (at finite cutoff) for rigid quadrangulations can be formulated as follows.
We define, somewhat informally, a \emph{rigid half-cylinder} $\mathbf{h}$ to be a one-ended infinite quadrangulation with simple boundary that, sufficiently far from the boundary, looks like a regular square tiling of a (half-infinite) Euclidean cylinder, see Figure~\ref{fig:trumpet}.
It is required to satisfy all requirements of a rigid quadrangulation, including being rooted at a convex corner, except that finitely many rays of $\mathbf{h}$ starting at a concave corner may have infinite length, instead of ending on a straight boundary vertex.
Then the half-infinite rays of $\mathbf{h}$ come in two types: those that spiral around the cylinder to the right and those that spiral to the left.
We denote by $\mathcal{H}^{(p,q)}$ for $p,q\geq 1$ the set of rigid half-cylinders with $p$ right-spiraling rays and $q$ left-spiraling rays, so that the circumference of the cylinder is $\sqrt{p^2+q^2}$.  

\begin{figure}[h]
    \centering
    \includegraphics[width=.45\linewidth]{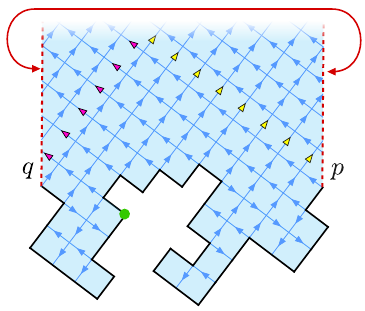}
    \caption{An example of a rigid half-cylinder in $\mathbf{h}\in\mathcal{H}^{(p,q)}$ for $p=8$ and $q=6$ and $n(\mathbf{h}) = 11$ non-root convex corners, cut open along a line (dashed red) perpendicular to the (in this case horizontal) circumference of the cylinder. The $p$ right-spiraling and $q$ left-spiraling rays are indicated by the yellow and pink arrows respectively.\label{fig:trumpet}}
\end{figure}

As will be explained in Section~\ref{sec:trumpets}, a rigid half-cylinder can be easily decomposed into C-type rigid quadrangulations.
Then Corollary~\ref{cor:multivariate} gives the following simple formula for the rigid half-cylinder generating function. 
\begin{corollary}\label{cor:trumpet}
    For $p,q\geq 1$, the (unrooted) generating function of rigid trumpets in $\mathcal{H}^{(p,q)}$ is given by 
    \begin{align}
        H^{(p,q)}(t) = \sum_{\mathbf{h}\in \mathcal{H}^{(p,q)}} \frac{t^{n(\mathbf{h})+1}}{n(\mathbf{h})+1} = \sum_{n\geq \max(p,q)} \frac{1}{n+1}\binom{2n-p+1}{n}\binom{2n-q+1}{n} R(t)^{n+1},
    \end{align}
    where we write $n(\mathbf{h})$ for the number of non-root convex corners of $\mathbf{h}$ and $R(t)$ is the same power series as in Theorem~\ref{thm:rigidquadenum}.
    The rooted generating function is therefore
    \begin{align}
        {H^{(p,q)}}{}'(t) = \sum_{\mathbf{h}\in \mathcal{H}^{(p,q)}} t^{n(\mathbf{h})} = \sum_{n\geq \max(p,q)} \binom{2n-p+1}{n}\binom{2n-q+1}{n} R(t)^n \, R'(t).
    \end{align} 
\end{corollary}

\noindent
For $p,q \geq 1$ and $0 < t \leq \tfrac{1}{4\pi}$ we consider the random rigid half-cylinder $\mathbf{H}\in \mathcal{H}^{(p,q)}$ with probability
\begin{align}
    \mathbb{P}_t^{(p,q)}(\mathbf{H}=\mathbf{h}) = \frac{1}{H^{(p,q)}(t)}\frac{t^{n(\mathbf{h})+1}}{n(\mathbf{h})+1},
\end{align}
called the \emph{(unrooted) Boltzmann $(p,q)$-rigid half-cylinder}.
In particular, when $q=1$ and $p\geq 1$ the Boltzmann $(p+1,1)$-rigid half-cylinder corresponds to sampling a rigid quadrangulation with 1-fold base of length $p$ from the partition function $[y^{p}]P(t,y)$ appearing in Corollary~\ref{cor:multivariate}, then gluing the base to a cylinder with a single left-spiraling ray in a canonical way and uniformly rooting the result on a convex corner.  

\begin{proposition}\label{prop:trumpetasymp}
    At the critical value $t = t_* = \frac{1}{4\pi}$, we have 
    \begin{align}
        H^{(p,q)}(t_*) \sim \frac{1}{\pi} \frac{2^{-p-q}}{p^2+q^2} \quad \text{as }p^2 + q^2 \to \infty. 
    \end{align}
    Furthermore, the random number $n(\mathbf{H})$ of convex corners in the critical Boltzmann $(p,q)$-rigid half-cylinder under $\mathbb{P}^{(p,q)}_{t_*}$ converges in distribution upon proper normalization to an inverse-exponential law,
    \begin{align}
        \frac{n(\mathbf{H}) \log n(\mathbf{H})}{p^2+q^2} \xrightarrow[p^2+q^2\to\infty]{\mathrm{(d)}} \frac{1}{\mathcal{E}(1)},
    \end{align} 
    where $\mathcal{E}(1)$ is an exponential random variable of mean $1$, and the convergence holds along any sequence of $(p,q) \in \Z^2_{>0}$ such that $p^2 + q^2 \to \infty$.
\end{proposition}

\noindent 
This result can be equivalently stated in terms of unnormalized measures.
For $p,q \geq 1$ and $\varepsilon > 0$ we define the \emph{renormalized boundary length} $\mathsf{Len}^{\varepsilon} : \mathcal{H}^{(p,q)} \to (0,\infty)$ by setting $\mathsf{Len}^{\varepsilon}(\mathbf{h}) = \tfrac{1}{2}\varepsilon^2 n(\mathbf{h})\log n(\mathbf{h})$.
Let $\mu^{(p,q)}$ be the discrete measure on $\mathcal{H}^{(p,q)}$ defined by $\mu^{(p,q)}(\{\mathbf{h}\}) = 2^{p+q} t_*^{n(\mathbf{h})+1}/(n(\mathbf{h})+1)$.
Then Proposition~\ref{prop:trumpetasymp} can be rephrased as a weak convergence of the pushforward measure $\mathsf{Len}^{\varepsilon}_{\,*} \mu^{(p,q)}$ on $(0,\infty)$: in a limit where $p^2+q^2 \to \infty$ and $\varepsilon\to 0$ such that $\varepsilon \sqrt{p^2+q^2} \to b > 0$ we have
\begin{align}
    \varepsilon^{-2}\,\mathsf{Len}^{\varepsilon}_{\,*} \mu^{(p,q)} \xrightarrow[\substack{p^2+q^2\to\infty\\\varepsilon\to 0}]{\mathrm{weak}} Z^{\text{half-cyl}}(\beta,b) \rmd\beta, \qquad Z^{\text{half-cyl}}(\beta,b) = \frac{1}{2\pi \beta^2} e^{-b^2/(2\beta)}.\label{eq:Zhalfcyl}
\end{align}
It is interesting to compare $Z^{\text{half-cyl}}(\beta,b)$, which may be tentatively interpreted as the continuum partition function at finite cutoff of (unmarked) half-cylinders with circumference $b$ and regularized boundary length $\beta$, to the JT trumpet partition function \eqref{eq:JTtrumpet}.
The resemblance is even closer to the half-cylinder amplitude quoted in \cite[Eq.~(8.56)]{Godet_New_2020} for $\widehat{\mathrm{CGHS}}$ gravity, which can be considered as a flat space analogue of JT gravity.

\subsection{Questions \& Discussion}

The presented results raise many natural questions, of which we highlight several here.

\begin{enumerate}[label=\Alph*.,font=\bfseries]
    \item \textbf{Description of inverse bijection.} Is there an explicit description of the inverse mapping $\Psi^{-1} : \mathcal{C} \to \mathcal{R}$ that does not involve the peeling exploration of a colorful quadrangulation? Note that in the construction of $\mathfrak{q} = \Psi(\mathbf{r})$, as a map drawn on the double of the rigid quadrangulation $\mathbf{r}$, the boundary of $\mathbf{r}$ corresponds to a simple closed curve on the sphere that visits each vertex of $\mathfrak{q}$ once. A description of $\Psi^{-1}$ could thus come in the form of an algorithm to identify this curve in a given colorful quadrangulation. This is akin to the \emph{Peano curves} involved in many bijections between families of planar maps and lattice walks, like Mullin's bijection \cite{Mullin_enumeration_1967} for spanning-tree-decorated maps, Sheffield's inventory accumulation \cite{Sheffield_Quantum_2016} for FK-decorated maps, and Bernardi's bijection for site-percolated triangulations \cite{Bernardi_Bijective_2007} (see also \cite[Section~2]{Gwynne_Mating_2023} for an overview of their role in discrete mating of trees). A hint in the direction of a Peano curve algorithm in our case can be found in Section~\ref{sec:ascentpath}, where it is shown that the portion of the boundary between the root corner and the next convex corner is related to a simple ascent path on the colorful quadrangulation. Combined with rerooting the rigid quadrangulation at the next convex corner, one can in principle iterate the procedure to identify the full boundary. However, the repeated rerooting becomes tedious rather quickly.
    \item \textbf{Generalized bijection.} Can $\Psi : \mathcal{R} \to \mathcal{C}$ be extended to a bijection between more general discrete flat disks and $\Z$-labeled maps? There are natural generalizations of the combinatorial families on both sides for which the enumeration remains tractable.
    For instance, in a forthcoming work we will address the enumeration of various families of tiled flat disks (including square-, triangle- and hexagon-tiled disks) without rigidity constraint via bijections with lattice walk.
    Another natural extension is to consider rigid quadrangulations with concave vertices of arbitrary degree larger or equal to $4$.
    On the other side, one can drop the colorful condition on the $\Z$-labeled quadrangulation and allow for faces with alternating label.
    The corresponding generating functions, also obtained by Bousquet-M\'elou, Elvey Price in \cite{Bousquet-Melou_generating_2020}, are very similar to those of colorful quadrangulations.
    Moreover, one can control the number of alternating faces \cite{Bousquet-Melou_Eulerian_2020,ElveyPrice_six_2023}, putting it in the realm of the six-vertex model, and the number of local maxima of the labeling \cite{Bousquet-Melou_Refined_2025}.\\
    With a little trick these enumeration problems can be interpreted in terms of rigid quadrangulations: by inserting in each alternating face with labels $(i,i+1)$ a new degree-$2$ vertex of label $i-1$, we can bijectively relate general $\Z$-labeled quadrangulations to colorful quadrangulations with a distinguished set of degree-$2$ local minima.
    By the dictionary in Theorem~\ref{thm:dictionary} these are in correspondence with rigid quadrangulations with a distinguished set of right-tangential convex corners of degree $2$.
    However, it is somewhat unsatisfactory that the definition of this family of disks depends strongly on the choice of root corner.
    This should be contrasted to $\mathcal{R}$ and $\mathcal{C}$, which are rerooting symmetric: $\mathbf{r}\in\mathcal{R}$ can be rerooted at each convex corner, while $\mathfrak{q} \in \mathcal{C}$ can be rerooted at each directed edge up to a unique affine relabeling to satisfy the root-label condition.
    So a better question is: does $\Psi$ admit a generalization to \emph{rerooting-symmetric} combinatorial families of disks and $\Z$-labeled maps? 

    \item\label{q:scalinglimit} \textbf{Disk area.} Does the area of a rigid quadrangulation have an interpretation at the level of colorful quadrangulations? What is the limiting distribution of the area for large rigid quadrangulations? In the dictionary of Theorem~\ref{thm:dictionary} two types of constituents of rigid quadrangulations are conspicuously absent: the inner vertices and faces of a rigid quadrangulation $\mathbf{r}$ do not have local counterparts at the level of the colorful quadrangulation $\mathfrak{q} = \Psi(\mathbf{r})$. Faces of $\mathbf{r}$ correspond to intersections of rows and columns, and therefore to pairs of an even and an odd edge of $\mathfrak{q}$ subject to a condition that does not enforce these edges to be nearby in $\mathfrak{q}$. Similarly, inner vertices of $\mathbf{r}$ correspond to special pairs of faces of $\mathfrak{q}$. This means in particular that it is challenging to enumerate rigid quadrangulations with control on the area. We refer to \cite[Section~4.3.3]{ferrari2024randomdisksconstantcurvature} for questions and predictions concerning the area distribution in the context of JT gravity.
    \item \textbf{Scaling limit.} Does the geometry of the uniform rigid quadrangulation of size $n$ admit a scaling limit as $n\to \infty$ in an appropriate topology? The asymptotics enumeration of Theorem~\ref{thm:rigidquadenum} and the distributional convergence of the size of Boltzmann rigid quadrangulations discussed in Section~\ref{sec:jt} hint at this possibility. 
    Given the similarity to $O(2)$ loop-decorated random planar maps \cite{Borot_recursive_2012,Budd_peeling_2018,Aidekon_scaling_2024,Kammerer_Gaskets_2024}, we expect the metric and the level lines of the uniform colorful quadrangulation to admit a scaling limit in terms of critical ($\gamma=2$) Liouville Quantum Gravity (LQG$_2$) together with an independent $\kappa = 4$ Conformal Loop Ensemble (CLE$_4$).
    The combination of LQG$_2$ and CLE$_4$ in turn is known to be closely related to a Brownian excursion in the half plane via a critical version of the mating of trees \cite{Aru_Mating_2021,Aru_Brownian_2023}.
    Since the latter involves an exploration of CLE$_4$ that can be understood as a continuum analogue of the peeling process underlying the bijection in Theorem~\ref{thm:rootedbijection}, it is natural to expect that the Brownian excursion can also be obtained in a scaling limit of the row-by-row exploration of the uniform rigid quadrangulation.
    A natural route is via comparison of the self-similar Markov trees (or signed growth-fragmentations) embedded in the exploration processes \cite{Bertoin_Martingales_2018,Bertoin_Self_2024} and the Brownian excursion \cite{Aidekon_Growth_2022} (see also \cite[Example~3.13]{Bertoin_Self_2024}).\\
    We finish this discussion with a conjecture based on a combination of these ideas. Consider a uniform rigid quadrangulation of size $n$ immersed in the square lattice $\Z^2 \subset \R^2$ with root corner at the origin (and aligned, say, with the upper-left quadrant). Denote by $\xi^{(n)} : [0,1] \to \R^2$ the random closed curve obtained by a constant-speed parametrization of its boundary. Then we expect that $\left(\xi^{(n)}(t) / \sqrt{n \log n}\right)_{t\in[0,1]}$ converges in distribution to a rotationally-invariant random closed curve $(\xi(t))_{t\in [0,1]}$. Moreover, its law is conjectured to be given in terms of a standard (one-dimensional) Brownian bridge $(X(t))_{t\in[0,1]}$ and an independent standard Brownian excursion $(\mathbf{e}(t))_{t\in [0,1]}$ via $\xi(t) = (X(t),Y(t))$ where
    \begin{align}
        Y(t) = \int_0^{\mathbf{e}(t)} \operatorname{sgn}\Big[ X\big(\inf\{s \in [t,1]: \mathbf{e}(s) = y \}\big) - X\big(\sup\{s \in [0,t]: \mathbf{e}(s) = y \}\big)\Big] \rmd y.
    \end{align}
    Note that the supposed rotational symmetry of $(\xi(t))_{t\in[0,1]}$ is not at all obvious from this construction.
\end{enumerate}

\subsection*{Acknowledgments}

We warmly thank Frank Ferrari for discussions on JT gravity at finite cutoff and for posing the challenge concerning the combinatorics of flat quadrangulations that inspired this pursuit. We are also grateful to Bart Zonneveld, Andrew Elvey Price and Nabin Shahid for useful discussions at various stages of the project. This work is supported by the VIDI programme with project number VI.Vidi.193.048, which is financed by the Dutch Research Council (NWO).

\section{Bijections with colorful quadrangulations}\label{sec:rowbyrow}

The bijection between rigid quadrangulations with $1$-fold base and colorful quadrangulations of the disk with labels $0,1,0,1,\ldots$ on the boundary is obtained by relating canonical explorations of both types of quadrangulations.
In the case of rigid quadrangulations this amounts to exploring row by row starting from the base, while for colorful quadrangulations we explore by peeling starting at the root.
In the next two subsection we will describe these explorations in detail, which involves formalizing the notion of an explored and unexplored part of the quadrangulations.

\subsection{Row-by-row exploration of rigid quadrangulations}\label{sec:rowbyrow}

\begin{figure}[ht]
    \centering
    \includegraphics[width=.7\linewidth]{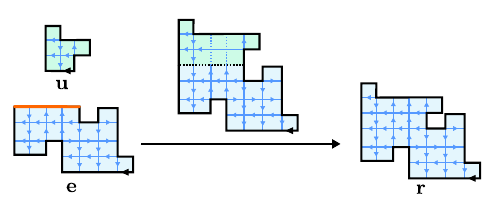}
    \caption{An example of gluing a base-$2$ rigid quadrangulation $\mathfrak{u}$ to the open side (in orange) of a partial rigid quadrangulation $\mathfrak{e}$, resulting in the rigid quadrangulation $\mathbf{r}$. Therefore, $\mathfrak{e}\subset\mathbf{r}$.
    \label{fig:gluingdisks}}
\end{figure}

Rigid quadrangulations with a $1$-fold base of length $p$ will be called \emph{base-$p$} rigid quadrangulations in the following.
To describe the exploration of the latter, we introduce a slight generalization: a \emph{partial} base-$p$ rigid quadrangulation $\mathbf{e}$ is a flat quadrangulation with zero or more distinguished horizontal sides (different from the base), called \emph{open sides}.
The inner edges of this quadrangulation are oriented along rays just like the rigid quadrangulations, except the rays are allowed to start and/or end at a straight vertex of an open side.
The \emph{size} of an open side is the number of rays starting at that side plus one.
To an open side $s$ of size $\ell$ one can naturally glue a base-$\ell$ rigid quadrangulation $\mathbf{u}$ as follows (see Figure~\ref{fig:gluingdisks} for an illustration).
Let us assume $s$ is a top side.
Then we identify the base of $\mathbf{u}$ with the side $s$ in such a way that the rays in $\mathbf{e}$ that start at $s$ align with the rays of $\mathbf{u}$ ending at its base.
As a consequence, the rays in $\mathbf{e}$ ending at $s$ touch the base of $\mathbf{u}$ somewhere in the interior of an edge.
To turn the result into a (partial) rigid quadrangulation, we extend each of these rays straight through the corresponding column of quadrangles until it meets the boundary, where it ends in a new straight corner.
Finally, the horizontal strips above and below $s$ are merged into a single horizontal strip.
In the case that $s$ is a bottom side, $\mathbf{u}$ is first rotated a half turn to make the base a top side and the gluing proceeds similarly.

We say that a partial base-$p$ rigid quadrangulation $\mathbf{e}$ is a \emph{submap} of a base-$p$ rigid quadrangulation $\mathbf{r}\in \mathcal{R}^{(p)}$, denoted $\mathbf{e}\subset \mathbf{r}$, if for each open side $s$ of $\mathbf{e}$ of size $\ell_s$ there exists a base-$\ell_s$ rigid quadrangulation $\mathbf{u}_s$ such that gluing $\mathbf{u}_s$ to $s$ for each side $s$ results in $\mathbf{r}$.  
Given $\mathbf{e}\subset\mathbf{r}$, it is easy to see that the rigid quadrangulations $\mathbf{u}_s$ are uniquely determined.
Hence, for fixed $\mathbf{e}$, the set $\{ \mathbf{r}\in\mathcal{R}^{(p)} : \mathbf{e}\subset \mathbf{r}\}$ is in bijection with tuples $(\mathbf{u}_s)_s$ of partial rigid quadrangulations of appropriate base length.
The notion of submap naturally extends to a pair of partial rigid quadrangulations, $\mathbf{e}\subset\mathbf{e}'$, where now the $\mathbf{u}_s$ are allowed to be partial rigid quadrangulations themselves.

\begin{lemma}\label{lem:minimalsubmaps}
    Every rigid quadrangulation $\mathbf{r}$ contains exactly one of the following submaps (see Figure~\ref{fig:peelingsteps}):
    \begin{itemize}
        \item $\mathbf{G}_{\ell,p-\ell-1}$ in which the base of length $p$ splits into two open sides of sizes $\ell$ and $p-\ell-1$ for some $\ell = 0,\ldots,p-1$. In the special cases $\ell=0$ or $\ell=p-1$ (or both) there will be a convex corner instead of an open side.
        \item $\mathbf{R}^p_\sigma$ where $\sigma$ is a finite word in the letters $\{\uparrow,\downarrow\}$. If the number of $\downarrow$ is $k\geq 0$ and the number of $\uparrow$ is $\ell\geq 0$ then it has an open top side of size $p+\ell$ and an open bottom side of size $k\geq 1$ to the right of the base. If $k=0$, there is no bottom side but a convex corner. The word $\sigma$ determines the orientations of the $\ell+k$ rightmost vertical rays (from left to right). 
        \item $\mathbf{L}^p_\sigma$ is similar to $\mathbf{R}^p_\sigma$, but the open side on the bottom appears to the left of the base and $\sigma$ determines the orientations of the $\ell+k$ leftmost vertical rays (from left to right).
    \end{itemize}
\end{lemma}

\begin{figure}[h]
    \centering
    \includegraphics[width=\linewidth]{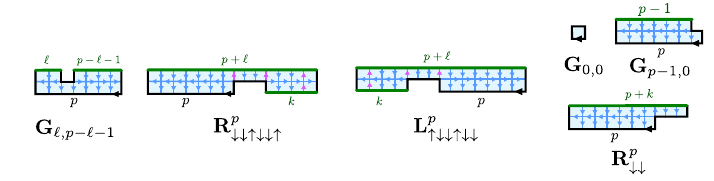}
    \caption{The three types of submaps of rigid quadrangulations described in Lemma~\ref{lem:minimalsubmaps}. The labels next to the open sides (green) indicate their size. The figures on the far right illustrate the boundary cases that have fewer than two open sides.
    \label{fig:peelingsteps}}
\end{figure}

\begin{proof}
    One may easily check that none of these are submaps of each other, so the options are mutually exclusive.
    Let us consider the endpoints $v_{\mathrm{L}}$ and $v_{\mathrm{R}}$ in $\mathbf{r}$ of the vertical edges adjacent to the base on the left and on the right.
    One of the following cases must apply:
    \begin{itemize}
        \item $v_{\mathrm{R}}$ and $v_{\mathrm{L}}$ are both convex corners. Then $\mathbf{r} = \mathbf{G}_{0,0}$ is the only option.
        \item Either $v_{\mathrm{L}}$ or $v_{\mathrm{R}}$ is a convex corner. Because of the downward ray next to this convex corner, it must be followed by a concave corner. The horizontal ray starting at the latter must end on a straight corner, which is the other one of $v_{\mathrm{L}}$ or $v_{\mathrm{R}}$. Hence, $\mathbf{G}_{\ell-1,0} \subset \mathbf{r}$ or $\mathbf{G}_{0,\ell-1} \subset \mathbf{r}$.
        \item $v_{\mathrm{R}}$ and $v_{\mathrm{L}}$ are both straight. Then the horizontal rays ending at $v_{\mathrm{R}}$ and $v_{\mathrm{L}}$ must start at a pair of adjacent concave corners. Hence, $\mathbf{G}_{k,\ell-k-1}\subset \mathbf{r}$ for some $k=1,\ldots,\ell-2$.
        \item $v_{\mathrm{R}}$ is a concave corner. Then the horizontal ray starting at $v_{\mathrm{R}}$ will end at a straight $v_{\mathrm{L}}$. Tracing the boundary from $v_{\mathrm{R}}$ to the right, the next corner at distance $d+1\geq 1$ can be
        \begin{itemize}
            \item a convex corner. Then $\mathbf{R}^{\ell}_{(\downarrow)^{d}} \subset \mathbf{r}$.
            \item a concave corner. The horizontal ray starting here will encounter an arbitrary sequence $\sigma$ of up and down rays before ending at a straight corner. Then $\mathbf{R}^{\ell}_{(\downarrow)^d\uparrow\sigma} \subset \mathbf{r}$.
        \end{itemize}
        \item $v_{\mathrm{L}}$ is a concave corner. Analogously to the previous case, $\mathbf{R}^{\ell}_\sigma \subset\mathbf{r}$ for some word $\sigma$. 
    \end{itemize}
\end{proof}

\begin{figure}[t]
    \centering
    \includegraphics[width=.9\linewidth]{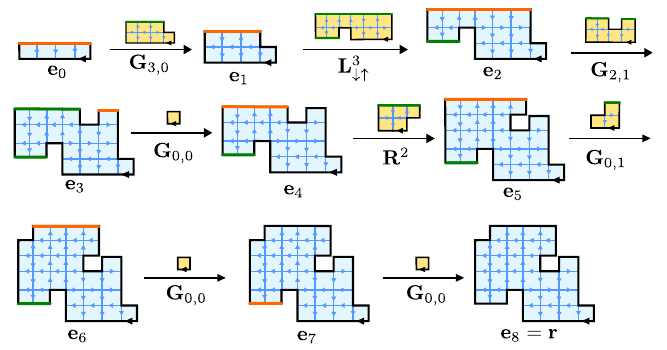}
    \caption{A row-by-row exploration of the rigid quadrangulation $\mathbf{r}\in\mathcal{R}^{(4)}_9$ shown in the bottom concave corner. The open sides of $\mathbf{e}_i$ are indicated in green and orange. Each step corresponds to gluing a partial rigid quadrangulation (shown in yellow) to the orange side.
    \label{fig:peeling}}
\end{figure}

A \emph{row-by-row exploration} or a rigid quadrangulation $\mathbf{r} \in \mathcal{R}^{(p)}_n$ is a sequence
\begin{align}
    \mathbf{e}_0 \subset \mathbf{e}_1 \subset \cdots \subset \mathbf{e}_{n-1}=\mathbf{r},
\end{align}
where $\mathbf{e}_0$ is the minimal submap of $\mathbf{r}$, which is the rectangle of width $p$ and height $1$ with the top side being open, and where $\mathbf{e}_{i+1}$ is obtained from $\mathbf{e}_i$ by gluing one of the partial quadrangulations of Lemma~\ref{lem:minimalsubmaps} to an open side.
See Figure~\ref{fig:peeling} for an example.

As a consequence of Lemma~\ref{lem:minimalsubmaps} every rigid quadrangulation $\mathbf{r}$ admits a row-by-row exploration.
At each peeling step the number of rows not adjacent to an open side increases by one, and therefore it takes exactly $n-1$ steps to fully explore $\mathbf{r}$.
Moreover, if we require that at each step one selects the open side that is encountered first in counterclockwise order starting from the root, then the row-by-row exploration is seen to be uniquely determined by $\mathbf{r}$.

\subsection{Peeling exploration of colorful quadrangulations}\label{sec:peeling}

A \emph{colorful quadrangulation (of the disk) with holes} $\mathfrak{e}$ is like an ordinary colorful quadrangulation except that it comes with a distinguished sequence $h_1,\ldots,h_m$ of faces (not including the root face), called \emph{holes}, each of which carries a \emph{marked edge} in its contour that is oriented in clockwise direction. 
The holes can be of arbitrary even degree, but they are required to be simple and they cannot touch each other (i.e.\ each vertex can be adjacent to at most one corner of a hole).
We require the labels around the hole to take exactly two values and to decrease along its marked edge.

For a hole $h$ of degree $2k$ there is a natural operation of gluing a colorful quadrangulation $\mathfrak{u}$ of perimeter $2k$ inside $h$.
If the labels of the marked edge $e$ are $r+1, r$, then we shift uniformly the labels of $\mathfrak{u}$ by $r$ and identify the boundary of $\mathfrak{u}$ with the contour of $h$, matching $e$ with the root edge.
Note that the boundary of $\mathfrak{u}$ is not necessarily simple, so the gluing operation may identify several vertices on the hole $h$ with the same vertex in $\mathfrak{u}$.
We say $\mathfrak{e}$ is a submap of a colorful quadrangulation $\mathfrak{q}$, denoted $\mathfrak{e}\subset\mathfrak{q}$, if for each hole $h$ of degree $2k$ there exists a colorful quadrangulation $\mathfrak{u}$ of perimeter $2k$ such that gluing $\mathfrak{u}$ in $h$ results in $\mathfrak{q}$.

\begin{figure}
    \centering
    \includegraphics[width=.9\linewidth]{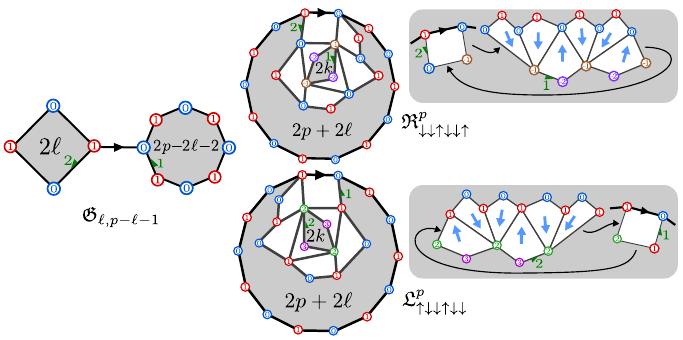}
    \caption{The minimal submaps of a colorful quadrangulation of perimeter $2p$. In these examples $p=7$, $k=2$, $\ell=4$. \label{fig:rigidquadtocolored}}
\end{figure}
\begin{lemma}\label{lem:minimalsubmapscq}
    Every colorful quadrangulation $\mathfrak{q}$ contains exactly one of the following submaps (see Figure~\ref{fig:rigidquadtocolored}):
    \begin{itemize}
        \item $\mathfrak{G}_{\ell,p-\ell-1}$ in which the root edge separates a hole $h_1$ of perimeter $2(p-\ell-1)$ and a hole $h_2$ of degree $2\ell$, for some $\ell=0,\ldots,p-1$. In the special cases $\ell=0$ or $\ell=p-1$ (or both) there is no hole and the root edge ends on a dangling vertex. The holes are marked at the unique clockwise oriented edge sharing the endpoint with the root edge (for $h_1$) or the origin with the root edge (for $h_2$).
        \item $\mathfrak{R}^p_\sigma$ where $\sigma$ is a finite word in the letters $\{\uparrow,\downarrow\}$.
        The root edge is adjacent to a quadrangle $f$ with labels $-1,0,1,0$. The pair of label-$(-1,0)$ sides are connected via a strip of quadrangles glued along label-$(-1,0)$ sides. Each $\downarrow$ (resp.\ $\uparrow$) corresponds to a quadrangle with labels $-1,0,1,0$ (resp.\ $-2,-1,0,-1$), in clockwise order along the strip. If the number of $\downarrow$ is $k\geq 0$ and the number of $\uparrow$ is $\ell\geq 0$, then there is a hole $h_1$ of degree $2k$ (surrounded by the strip) marked at the clockwise edge starting at the vertex of $f$ with label $-1$, and a hole $h_2$ of degree $2p+2\ell$ (exterior to the strip) marked at the remaining edge of $f$.
        \item $\mathfrak{L}^p_\sigma$ is similar to $\mathfrak{R}^p_\sigma$, but the root edge is adjacent to a quadrangle $f$ with labels $0,1,2,1$ and the label-$(1,2)$ sides are connected by a strip of quadrangles. Each $\downarrow$ (resp.\ $\uparrow$) corresponds to a quadrangle with labels $0,1,2,1$ (resp.\ $1,2,3,2$), in clockwise order along the strip. If the number of $\downarrow$ is $k\geq 0$ and the number of $\uparrow$ is $\ell\geq 0$, there is a hole $h_1$ of degree $2p+2\ell$ (exterior to the strip) marked at the remaining edge of $f$, and a hole $h_2$ of degree $2k$ (surrounded by the strip) marked at the clockwise edge ending at the vertex of $f$ with label $2$.
    \end{itemize}
\end{lemma}
\begin{proof}
    We consider the possibilities for the root edge of $\mathfrak{q}$. Either it is identified with another side of the root face, or it is adjacent to a quadrangle $f$. 
    In the first case, $\mathfrak{G}_{\ell,p-\ell-1} \subset \mathfrak{q}$ for some $\ell=0,\ldots,p-1$.
    In the second case, because of the colorful condition, the quadrangle $f$ must have labels $-1,0,1,0$ or $0,1,2,1$.
    If $f$ has labels $-1,0,1,0$, we may follow the path on the dual map that crosses only $(-1,0)$-edges of $\mathfrak{q}$ starting at $f$ in the direction of the $(-1,0)$-edge next to the root.
    Due to the colorful condition and the fact that the boundary does not contain $(-1,0)$-edges, this cycle must return to $f$ and passes only through quadrangles with labels $-1,0,1,0$ or $-2,-1,0,-1$.
    If we record the latter as a sequence of $\downarrow$ and $\uparrow$ respectively, we obtain a word $\sigma$.
    Hence, $\mathfrak{R}^p_\sigma \subset \mathfrak{q}$. 
    If instead $f$ has labels $0,1,2,1$, we may follow the path on the dual map that crosses only $(1,2)$-edges of $\mathfrak{q}$ starting at $f$ in the direction of the $(1,2)$-edge opposite to the root.
    In this case, the cycle must return to $f$ and passes only through quadrangles with labels $0,1,2,1$ or $1,2,3,2$.
    Recording these with a $\downarrow$ and $\uparrow$ in $\sigma$ respectively, we have $\mathfrak{L}^p_\sigma \subset \mathfrak{q}$. 
\end{proof}

\begin{figure}[h]
    \centering
    \includegraphics[width=\linewidth]{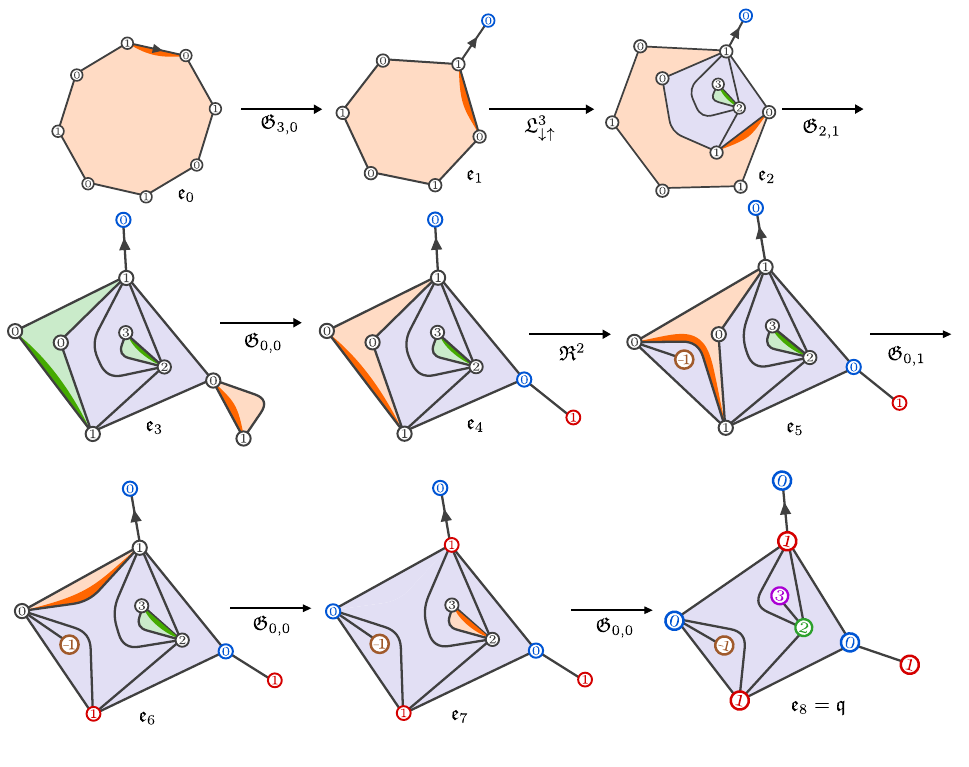}
    \caption{The peeling exploration of the colorful quadrangulation $\mathfrak{q}\in\mathcal{C}^{(4)}_9$ shown in the bottom concave corner, corresponding to row-by-row exploration of the rigid quadrangulation in Figure~\ref{fig:peeling}. The holes of $\mathfrak{e}_i$ are indicated in green and orange. Each peeling step corresponds to gluing a colorful quadrangulation with holes inside the orange hole.
    \label{fig:peelinglq}}
\end{figure}

The \emph{peeling exploration} or a colorful quadrangulation $\mathfrak{q} \in \mathcal{C}^{(p)}_n$ is the sequence
\begin{align}\label{eq:peelingexplor}
    \mathfrak{e}_0 \subset \mathfrak{e}_1 \subset \cdots \subset \mathfrak{e}_{n-1}=\mathfrak{q},
\end{align}
where $\mathfrak{e}_0$ is the minimal submap of $\mathfrak{q}$, consisting solely of a root face and a hole both of degree $2p$, and where $\mathfrak{e}_{i+1}$ is obtained from $\mathfrak{e}_i$ by
gluing one of the colorful quadrangulations with holes of Lemma~\ref{lem:minimalsubmaps} inside the first hole $h_1$ of $\mathfrak{e}_i$.
Note that the ordering of the holes is maintained appropriately, meaning that if $\mathfrak{e}_i$ has holes $h_1,h_2,\ldots$, and $\mathfrak{G}_{\ell,p-\ell-1}$ with holes $\tilde{h}_1,\tilde{h}_2$ is glued in the first hole $h_1$, then $\mathfrak{e}_{i+1}$ has its holes ordered as $\tilde{h}_1,\tilde{h}_2,h_2,\ldots$.
As a consequence of the lemma, this peeling exploration always exists and is unique.
See Figure~\ref{fig:peelinglq} for an example.
At each step the number of even edges that are not adjacent to a hole increases by one.

\subsection{First statement of the bijection}

We define $\Psi^{(p)}: \mathcal{R}^{(p)} \to \mathcal{C}^{(p)}$ as follows. 
Let $\mathbf{r} \in \mathcal{R}^{(p)}$ be a rigid quadrangulation and $\mathbf{e}_0 \subset \mathbf{e}_1 \subset \cdots \subset \mathbf{e}_{n-1}$ its row-by-row  exploration.
Set $\mathfrak{e}_0$ to be the minimal colorful quadrangulation with hole of perimeter $2p$.
Inductively, if $\mathbf{e}_{i+1}$ is obtained from $\mathbf{e}_{i}$ by gluing $\mathbf{G}_{k,\ell-k-1}$, $\mathbf{R}^\ell_\sigma$ or $\mathbf{L}^\ell_\sigma$ to the first open side of $\mathbf{e}_{i}$, then $\mathfrak{e}_{i+1}$ is obtained from $\mathfrak{e}_i$ by gluing the corresponding $\mathfrak{G}_{k,\ell-k-1}$, $\mathfrak{R}^\ell_\sigma$ or $\mathfrak{L}^\ell_\sigma$ into the first hole.
Then we let $\Psi^{(p)}(\mathbf{r}) = \mathfrak{e}_{n-1}$.

\begin{theorem}\label{thm:bijection}
    For each $p\geq 1$, the mapping $\Psi^{(p)} : \mathcal{R}^{(p)} \to \mathcal{C}^{(p)}$ is a bijection.
\end{theorem}

\begin{figure}[h]
    \centering
    \includegraphics[width=.6\linewidth]{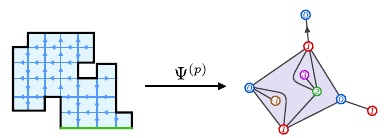}
    \caption{The result of relating the row-by-row exploration of Figure~\ref{fig:peeling} to the peeling exploration of Figure~\ref{fig:peelinglq}. In this example $p=4$.}
\end{figure}

\begin{proof}
    We first convince ourselves that the mapping is well-defined.
    It follows directly by induction on $i$ that $\mathfrak{e}_i$ is a colorful quadrangulation of perimeter $2p$ with holes whose half-perimeters match precisely the sizes of open sides of $\mathbf{e}_i$ in order.
    Since $\mathbf{e}_{n-1} = \mathbf{r}$ has no open sides, $\mathfrak{q} = \Psi^{(p)}(\mathbf{r}) = \mathbf{e}_{n-1}$ is a colorful quadrangulation without holes, as desired.
    Moreover, by construction, $\mathfrak{e}_0 \subset \mathfrak{e}_1 \subset \cdots \subset \mathfrak{q}$ is the peeling exploration of $\mathfrak{q}$.
    Therefore, the inverse mapping is constructed by a similar inductive procedure: if $\mathfrak{e}_{i+1}$ is obtained from $\mathfrak{e}_{i}$ by gluing $\mathfrak{G}_{k,\ell-k-1}$, $\mathfrak{R}^\ell_\sigma$ or $\mathfrak{L}^\ell_\sigma$ inside the first hole of $\mathfrak{e}_{i}$, then $\mathbf{e}_{i+1}$ is obtained from $\mathbf{e}_i$ by gluing the corresponding $\mathbf{G}_{k,\ell-k-1}$, $\mathbf{R}^\ell_\sigma$ or $\mathbf{L}^\ell_\sigma$ to the first open side.
    Then the rigid quadrangulation $\mathbf{r}$ is taken to be $\mathbf{e}_{n-1}$.
    It should be clear that this is again a well-defined mapping and that it is the inverse to $\Psi^{(p)}$.
\end{proof}

\subsection{Direct description of the bijection}

The bijection just described is not unique, because different choices of marked edges in the holes of the minimal submaps of Lemma~\ref{lem:minimalsubmapscq} would have led to different bijections.
Readers familiar with peeling processes will recognize this as the freedom in choosing a \emph{peeling algorithm} without affecting the enumeration (see \cite{Curien_Peeling_2023} for a comprehensive overview).
However, as we will see, our particular choice of peeling algorithm leads to a bijection that respects the reflection symmetry of rigid quadrangulations. 

\begin{proposition}\label{prop:bijectionequivalence}
For each $p\geq 1$, we have $\Psi^{\mathrm{b}} |_{\mathcal{R}^{(p)}} = \Psi^{(p)}$.
\end{proposition}

In fact, we will prove a slightly stronger result because $\mathfrak{q}=\Psi^{\mathrm{b}}(\mathbf{r})$ provides not just a planar map but also a canonical embedding in the double of $\mathbf{r}$. 
In particular, we can locate the vertices in the boundary of $\Psi^{\mathrm{b}}(\mathbf{r})$ starting from the edges in the base of $\mathbf{r}$. 
Let us order these edges from right to left starting from the root.
Then to the $i$th edge, $i = 1,\ldots,p$, we associate a pair of paths $\vec{\gamma}_i(\mathbf{r}), \cev{\gamma}_i(\mathbf{r})$, both starting in the interior of that edge and ending on a convex corner of $\mathbf{r}$.
To this end we consider the column of $\mathbf{r}$ that ends on the $i$th edge.
The path $\vec{\gamma}_i(\mathbf{r})$ is obtained by extending the middle section to end on a convex vertex by alternating the right and upward direction whenever hitting the boundary.
In other words, this corresponds to a one-sided version of the edge drawing algorithm discussed in Section~\ref{sec:bijectionintro} in which the column is transversally oriented to the right.
Similarly, $\cev{\gamma}_i(\mathbf{r})$ is the path corresponding to this column with transverse orientation to the left.
It should be clear that the paths $\vec{\gamma}_i(\mathbf{r})$ and $\cev{\gamma}_i(\mathbf{r})$ are disjoint if we start $\vec{\gamma}_i(\mathbf{r})$ just to the right of $\cev{\gamma}_i(\mathbf{r})$, see Figure~\ref{fig:doublepaths}b.

\begin{figure}[t]
    \centering
    \includegraphics[width=.8\linewidth]{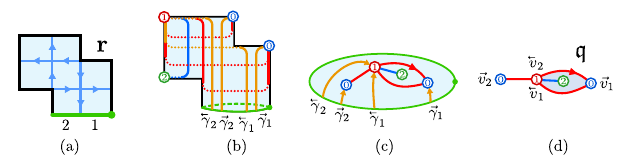}
    \caption{\label{fig:doublepaths} (a) The rigid quadrangulation $\mathbf{r}\in\mathcal{R}_2$ with the right-to-left ordering of the edges in its base. (b) We think of the double of $\mathbf{r}$ as a topological disk with boundary corresponding to the base. The paths $\vec{\gamma}_i(\mathbf{r})$ and $\vec{\gamma}_i(\mathbf{r})$ are indicated in orange. (c) A topologically equivalent illustration of the double. (d) The colorful quadrangulation $\mathfrak{q} = \Psi^{(2)}(\mathbf{r})$ with the vertices $\vec{v}_i(\mathfrak{q})$ and $\vec{v}_i(\mathfrak{q})$ along its boundary indicated.}
\end{figure}

Let us further denote the vertices encountered in clockwise order around the boundary of $\mathfrak{q}$ by 
\begin{align*}
\vec{v}_1(\mathfrak{q}), \cev{v}_1(\mathfrak{q}), \vec{v}_2(\mathfrak{q}), \cev{v}_2(\mathfrak{q}),\ldots,\vec{v}_p(\mathfrak{q}), \cev{v}_p(\mathfrak{q}),
\end{align*}
such that the root points from $\cev{v}_p(\mathfrak{q})$ to $\vec{v}_1(\mathfrak{q})$.
In other words, $\vec{v}_i(\mathfrak{q})$ is the $i$th label-$0$ vertex and $\cev{v}_i(\mathfrak{q})$ is the $i$th label-$1$ vertex encountered in clockwise order starting from the root, see Figure~\ref{fig:doublepaths}d.

\begin{proposition}\label{prop:extendbijequivalence}
    For each $p\geq 1$ and $\mathbf{r}\in\mathcal{R}^{(p)}$, the map $\mathfrak{q}=\Psi^{\mathrm{b}}(\mathbf{r})$ drawn on the double of $\mathbf{r}$ satisfies
    \begin{enumerate}[label = (\roman*)]
        \item $\mathfrak{q}$ is isomorphic as planar map to $\Psi^{(p)}(\mathbf{r})$;
        \item the paths $\vec{\gamma}_1(\mathbf{r}), \cev{\gamma}_1(\mathbf{r}), \vec{\gamma}_2(\mathbf{r}), \cev{\gamma}_2(\mathbf{r}),\ldots,\vec{\gamma}_p(\mathbf{r}), \cev{\gamma}_p(\mathbf{r})$ are mutually disjoint and disjoint from $\mathfrak{q}$ (except at their endpoints) and end on the vertices $\vec{v}_1(\mathfrak{q}), \cev{v}_1(\mathfrak{q}), \vec{v}_2(\mathfrak{q}), \cev{v}_2(\mathfrak{q}),\ldots,\vec{v}_p(\mathfrak{q}), \cev{v}_p(\mathfrak{q})$.
    \end{enumerate}
\end{proposition}

\begin{proof}[Proof of Proposition~\ref{prop:bijectionequivalence}]
    This is just (i) of Proposition~\ref{prop:extendbijequivalence}.
\end{proof}

\noindent
The advantage of this stronger statement is that we can prove it by induction.

\begin{proof}[Proof of Proposition~\ref{prop:extendbijequivalence}]
    We perform induction on the number of edges of $\mathfrak{q}$ while $p\geq 1$ remains arbitrary. 
    For the initial case, there is a unique $\mathbf{r}$ for which $\mathfrak{q}$ has a single edge, and the statement is easily verified by inspection (Figure~\ref{fig:doublepaths-basecase}).
    \begin{figure}[h!]
        \centering
        \includegraphics[width=.35\linewidth]{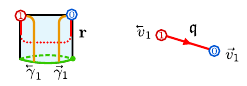}
        \caption{\label{fig:doublepaths-basecase}}
    \end{figure}

    For the induction step, let $\mathfrak{q}=\Psi^{\mathrm{b}}(\mathbf{r})$ and assume that the statement of the proposition holds for every rigid quadrangulation $\mathbf{r}'$ such that $\Psi^{\mathrm{b}}(\mathbf{r}')$ has fewer edges than $\mathfrak{q}$.
    Let $\mathbf{e}$ be the unique submap of Lemma~\ref{lem:minimalsubmaps} such that $\mathbf{e} \subset \mathbf{r}$ and let $\mathfrak{e}$ be the corresponding submap of Lemma~\ref{lem:minimalsubmapscq}.
    The rigid quadrangulation $\mathbf{r}$ is recovered from $\mathbf{e}$ by gluing a rigid quadrangulation $\mathbf{u}_j$ of base $\ell_j$ to the $j$th open side of $\mathbf{e}$ for the index $j$ ranging over the (at most two) open sides.
    Let $\mathfrak{u}_j = \Psi^{\mathrm{b}}(\mathbf{u}_j)$ be the map drawn on the double of $\mathbf{u}_j$.
    By Theorem~\ref{thm:bijection}, $\mathfrak{e} \subset \Psi^{(p)}(\mathbf{r})$ and the maps to be glued in the holes of $\mathfrak{e}$ to recover $\Psi^{(p)}(\mathbf{r})$ are $\Psi^{(\ell_j)}(\mathbf{u}_j)$.
    Since $\mathfrak{u}_j$ has fewer edges than $\mathfrak{q}$, by our assumption $\mathfrak{u}_j$ agree with $\Psi^{(\ell_j)}(\mathbf{u}_j)$ as planar map.
    We need to show that $\mathfrak{q}$ corresponds to the appropriate gluing of $\mathfrak{u}_j$ inside the holes of $\mathfrak{e}$, and that the paths of (ii) can be drawn in a disjoint fashion.
    To this end we inspect $\mathbf{e}$ case by case.

    \paragraph{Case $\mathbf{e} = \mathbf{G}_{\ell_1,\ell_2}$ for $p = \ell_1 + \ell_2 +1$:} 
    \begin{figure}[h!]
        \centering
        \includegraphics[width=\linewidth]{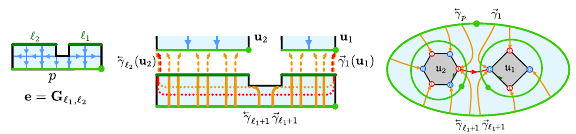}
        \caption{\label{fig:doubleGstep}}
    \end{figure}
    The root edge of $\mathfrak{q}$ corresponds to the row above the base of $\mathbf{r}$ (shown in red in Figure~\ref{fig:doubleGstep}).
    It turns upward on both ends of the row.
    In the right direction, if $\ell_1 = 0$, it ends on the top-right corner of $\mathbf{e}$ corresponding to $\vec{v}_1(\mathfrak{q})$, otherwise it traces $\vec{\gamma}_1(\mathbf{u}_1)$ and thus ends at $\vec{v}_1(\mathfrak{u}_1)$. 
    In the left direction, if $\ell_2 = 0$, it ends on the top-left corner of $\mathbf{e}$ corresponding to $\cev{v}_p(\mathfrak{q})$, otherwise it traces $\cev{\gamma}_{\ell_2}(\mathbf{u}_2)$ and thus ends at $\cev{v}_{\ell_2}(\mathfrak{u}_2)$.
    This agrees precisely with $\Psi^{(p)}(\mathbf{r})$ showing that it is isomorphic to $\mathfrak{q}$.
    We can trace the paths $\vec{\gamma}_i(\mathbf{r}),\cev{\gamma}_i(\mathbf{r})$ explicitly:
        \begin{itemize}
            \item For $i = 1, \ldots, \ell_1$, the path $\vec{\gamma}_i(\mathbf{r})$ continues as $\vec{\gamma}_i(\mathbf{u}_1)$ and thus ends on $\vec{v}_i(\mathfrak{u}_1)$, while $\cev{\gamma}_i(\mathbf{r})$ continues as $\cev{\gamma}_i(\mathbf{u}_1)$ and ends on $\cev{v}_i(\mathfrak{u}_1)$.
            \item For $i = 1, \ldots, \ell_2$, the path $\vec{\gamma}_{\ell_1+i+1}(\mathbf{r})$ continues as $\vec{\gamma}_i(\mathbf{u}_2)$ ending on $\vec{v}_i(\mathfrak{u}_2)$, while $\cev{\gamma}_{\ell_1+i+1}(\mathbf{r})$ continues as $\cev{\gamma}_i(\mathbf{u}_2)$ ending at $\cev{v}_i(\mathfrak{u}_2)$.
            \item The path $\vec{\gamma}_{\ell_1+1}(\mathbf{r})$ turns right and, if $\ell_1>0$, upwards again tracing $\vec{\gamma}_1(\mathbf{u}_1)$ thus ending on $\vec{v}_1(\mathfrak{u}_1)$. If $\ell_1 = 0$, $\vec{\gamma}_{\ell_1+1}(\mathbf{r})$ ends on the top-right corner $\vec{v}_1(\mathfrak{q})$ of $\mathbf{e}$.
            \item The path $\cev{\gamma}_{\ell_1+1}(\mathbf{r})$ turns left and, if $\ell_2>0$, upwards again tracing $\cev{\gamma}_{\ell_2}(\mathbf{u}_2)$ thus ending on $\cev{v}_{\ell_2}(\mathfrak{u}_2)$. If $\ell_2 = 0$, $\cev{\gamma}_{\ell_1+1}(\mathbf{r})$ ends on the top-left corner $\cev{v}_p(\mathfrak{q})$ of $\mathbf{e}$.
        \end{itemize}
    We observe that for each $i=1,\ldots,p$, $\vec{\gamma}_i(\mathbf{r})$ ends at $\vec{v}_i(\mathfrak{q})$ and $\cev{\gamma}_i(\mathbf{r})$ ends at $\cev{v}_i(\mathfrak{q})$. Moreover, the paths are disjoint.

    \begin{figure}[h]
    \centering
    \includegraphics[width=\linewidth]{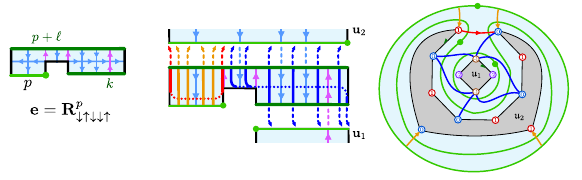}

    \vspace{-4mm}
    \caption{Example of a step $\mathbf{R}_{\sigma}^p$ with $n^\uparrow = 2$, $n^\downarrow = 3$.\label{fig:doubleRstep} }
    \end{figure} 
    \paragraph{Case $\mathbf{e} = \mathbf{R}_{\sigma}^p$:} Let $n^{\uparrow}$ and $n^{\downarrow}$ be the number of $\uparrow$ respectively $\downarrow$ in the word $\sigma$, see Figure~\ref{fig:doubleRstep} 
    The (even) root edge of $\mathfrak{q}$ corresponds again to the row above the base.
    In the right direction, it turns up and traces $\vec{\gamma}_{n^{\downarrow}+1}(\mathbf{u}_1)$ ending at $\vec{v}_{n^{\downarrow}+1}(\mathfrak{u}_1)$, while in the left direction, it turns up and traces $\cev{\gamma}_{p+n^{\downarrow}}(\mathbf{u}_1)$ ending at $\cev{v}_{p+n^{\downarrow}}(\mathfrak{u}_1)$.
    In addition to the edges of $\mathfrak{u}_1$ and $\mathfrak{u}_2$, $\mathfrak{q}$ contains $\ell+k+1$ odd edges corresponding to the columns of $\mathbf{e}$ that do not end on the base.
    For each such column $\chi$ an edge is drawn as follows.
    Let $n^\uparrow_{\mathrm{L}}$, $n^\uparrow_{\mathrm{R}}$ be the number of $\uparrow$ to the left respectively right of the $\chi$, and definite $n^\downarrow_{\mathrm{L}}$, $n^\downarrow_{\mathrm{R}}$ analogously.
    Then upwards the edge traces $\vec{\gamma}_{n^\downarrow_{\mathrm{R}}+1}(\mathbf{u}_2)$ ending on $\vec{v}_{n^\downarrow_{\mathrm{R}}+1}(\mathfrak{u}_2)$.
    If $n^\uparrow = 0$, downwards it turns right and ends on the bottom-right corner of $\mathbf{e}$.
    Otherwise, if $n^\uparrow_{\mathrm{L}} = 0$ and $n^\uparrow_{\mathrm{R}} > 0$, it turns right, hits the right boundary, turns down again and traces $\cev{\gamma}_{n^\uparrow}(\mathbf{u}_1)$ ending on $\cev{v}_{n^\uparrow}(\mathfrak{u}_1)$ (note that $\mathbf{u}_1$ was rotated by $180$ degrees before gluing to $\mathbf{e}$!).
    Finally, if $n^\uparrow_{\mathrm{L}} > 0$ and $n^\uparrow_{\mathrm{R}} > 0$, then the edge continues downward tracing $\cev{\gamma}_{n^\uparrow_{\mathrm{L}}}(\mathbf{u}_1)$ ending on $\cev{v}_{n^\uparrow_{\mathrm{L}}}(\mathfrak{u}_1)$.
    One may check that this corresponds precisely to the edges in $\Psi^{(p)}(\mathbf{r})$ resulting from gluing $\mathfrak{u}_i$ in the holes of $\mathfrak{R}_\sigma^p$, showing that $\Psi^{(p)}(\mathbf{r})$ is isomorphic to $\mathfrak{q}$.

    For $i=1,\ldots,p$, the path $\vec{\gamma}_i(\mathbf{r})$ continues upwards and trace $\vec{\gamma}_{i+n^\downarrow+1}(\mathbf{u}_2)$ ending on $\vec{v}_{i+n^\downarrow+1}(\mathfrak{u}_2) = \vec{v}_i(\mathfrak{q})$. Similarly, $\cev{\gamma}_i(\mathbf{r})$ traces $\cev{\gamma}_{i+n^\downarrow+1}(\mathbf{u}_2)$ and ends on $\cev{v}_{i+n^\downarrow+1}(\mathfrak{u}_2) = \cev{v}_i(\mathfrak{q})$.

    \begin{figure}[h]
    \centering
    \includegraphics[width=\linewidth]{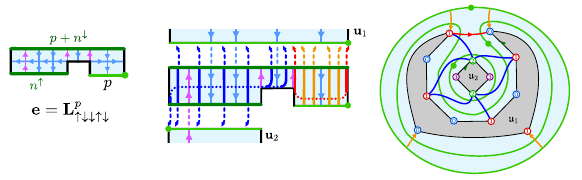}
    \caption{A step $\mathbf{L}^p_{\sigma}$ that is the mirrored version of Figure~\ref{fig:doubleRstep}.\label{fig:doubleLstep} }
    \end{figure} 
    \paragraph{Case $\mathbf{e} = \mathbf{L}_{\sigma}^p$:} This situation is completely analogous to $\mathbf{L}_{\sigma}^p$, but mirrored horizontally and one has to reverse any orientation-dependent ordering, see Figure~\ref{fig:doubleLstep}. We leave it to the reader to verify, with help of the figure below, that again $\Psi^{(p)}(\mathbf{r}) = \mathfrak{q}$ and that the paths $\vec{\gamma}_i(\mathbf{r}),\cev{\gamma}_i(\mathbf{r})$ end on the appropriate vertices $\vec{v}_{i}(\mathfrak{q}),\cev{v}_{i}(\mathfrak{q})$.

    This completes the proof of the proposition by induction.
\end{proof}

\subsection{Zipping the root}\label{sec:zipping}
In particular, we have established Theorem~\ref{thm:basebijection} for $k = 1$ and $\ell_1 = p \geq 1$. 
As promised, let us deduce Theorem~\ref{thm:rootedbijection} from the special case $k = \ell_1 = 1$.

\begin{figure}[h]
    \centering
    \includegraphics[width=.9\linewidth]{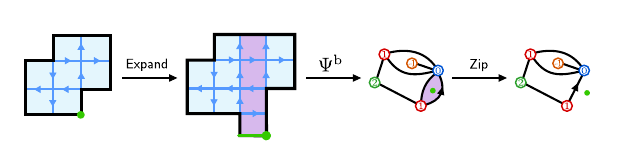}
    \caption{\label{fig:zipping}}
\end{figure}

\begin{proof}[Proof of Theorem~\ref{thm:rootedbijection}]
    Let us denote by $\hat{\mathcal{R}}^{(1)} \subset \mathcal{R}^{(1)}$ the base-$1$ rigid quadrangulations in which the top of the quadrangle above the base is an inner edge oriented to the right.
    Let also $\hat{\mathcal{C}}^{(1)} \subset \mathcal{C}^{(1)}$ be the colorful quadrangulations of perimeter $2$ in which the root edge is adjacent to a quadrangle with labels $(0,1,2,1)$.
    We will verify that
    \begin{align*}
        \Psi = \mathsf{Zip} \circ \Psi^{\mathrm{b}}\big|_{\hat{\mathcal{R}}^{(1)}} \circ\mathsf{Expand},
    \end{align*}
    where $\mathsf{Zip} : \hat{\mathcal{C}}^{(1)} \to \mathcal{C}$ and $\mathsf{Expand} : \mathcal{R} \to \hat{\mathcal{R}}^{(1)} $ are two natural bijections defined as follows (see Figure~\ref{fig:zipping}).
    For $\mathfrak{q}\in\hat{\mathcal{C}}^{(1)}$, we define $\mathsf{Zip}(\mathfrak{q})$ to be the colorful quadrangulation obtained from $\mathfrak{q}$ by merging both edges on the boundary into a single root edge.
    Given a rigid quadrangulation $\mathbf{r} \in \mathcal{R}$, we define $\mathsf{Expand}(\mathbf{r})$ to be the rigid quadrangulation obtained by splitting the column to the left of the root corner into two columns, attaching a new quadrangle at the bottom, and taking the bottom-right corner to be the new root.
    By construction, the top of the quadrangle above the base is oriented to the right.
    That $\mathsf{Zip}$ and $\mathsf{Expand}$ are bijections can be readily checked. 
    The subsets $\hat{\mathcal{R}}^{(1)} \subset \mathcal{R}^{(1)}$ and $\hat{\mathcal{C}}^{(1)} \subset \mathcal{C}^{(1)}$ can be understood as the rigid quadrangulations, respectively colorful quadrangulations, for which the first step in the exploration is of type $\mathbf{L}_\sigma^1$, respectively $\mathcal{L}_\sigma^1$.
    By Theorem~\ref{thm:bijection} and Proposition~\ref{prop:bijectionequivalence}, $\Psi^{\mathrm{b}}\big|_{\hat{\mathcal{R}}^{(1)}} = \Psi^{(1)}\big|_{\hat{\mathcal{R}}^{(1)}}$ is indeed a bijection.

    Now one may verify that $\Psi^{\mathrm{b}}(\mathbf{r}')$ drawn on the double of $\mathbf{r}'=\mathsf{Expand}(\mathbf{r})$ is (a slightly expanded version of) $\Psi(\mathbf{r})$ drawn on the double of $\mathbf{r}$, except that there is a new root edge corresponding to the row above the base of $\mathbf{r}'$, see Figure~\ref{fig:zipping-relation}.  
    \begin{figure}[h!]
        \centering
        \includegraphics[width=.5\linewidth]{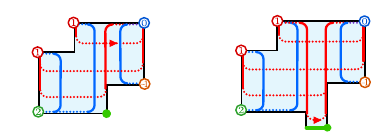}
        \caption{\label{fig:zipping-relation}}
    \end{figure}

    With the knowledge of the location of the root corner, $\mathsf{Zip}$ amounts to deleting this edge and rooting the resulting quadrangulation on the unique edge pointing from $1$ to $0$ with the root corner contained in the face on its right. 
    But that is precisely $\Psi(\mathbf{r})$, so $\Psi(\mathbf{r}) = \mathsf{Zip} \circ \Psi^{\mathrm{b}}\big|_{\hat{\mathcal{R}}^{(1)}} \circ\mathsf{Expand}(\mathbf{r})$.
    In particular, $\Psi$ is a bijection.
\end{proof}

\subsection{Dictionary}

With the bijection $\Psi : \mathcal{R} \to \mathcal{C}$ in place we are now ready to relate some natural statistics on both sides. 

\begin{proof}[Proof of Theorem~\ref{thm:dictionary}]
    For convenience, we reproduce here the table of correspondences that we need to prove.
\begin{center}
    \begin{tabular}{r|l|l}
        & Rigid quadrangulation $\mathbf{r}$ & Colorful quadrangulation $\mathfrak{q}$ \\\hline
        \ref{it:vertex} & non-root convex corner & vertex \\
        \ref{it:evenedge} & row & even edge \\
        \ref{it:oddedge} & column & odd edge \\
        \ref{it:face} & root-corner or concave corner & face\\
        \ref{it:rightconvex} & right-tangential convex corner of degree $\ell$ & local minimum of degree $\ell$ \\
        \ref{it:leftconvex} & left-tangential convex corner of degree $\ell$ & local maximum of degree $\ell$ \\
        \ref{it:rightside} & right-tangential side & decreasing level line \\
        \ref{it:leftside} & left-tangential side & increasing level line \\\hline
    \end{tabular}
\end{center}

    \noindent
    We discuss each of the elements of $\mathfrak{q} = \Psi(\mathbf{r})$ in turn:

    \ref{it:vertex}, \ref{it:evenedge} \& \ref{it:oddedge}. These follow directly from the definition of $\Psi$.

    \ref{it:face}. In a rigid quadrangulation $\mathbf{r}$ with $n+1$ convex corners there are precisely $n-3$ concave corners, while in a colorful quadrangulation with $n$ vertices there are $n-2$ faces.
    Therefore it is sufficient to show that each face of $\mathfrak{q}$ contains a concave corner or the root corner.
    To this end, let us consider the expanded version $\mathbf{r}' = \mathsf{Expand}(\mathbf{r}) \in \mathcal{R}_1$. 
    Now we wish to show that each face of $\mathfrak{q}' = \Psi^{(1)}(\mathbf{r}')$, except the root face, contains a concave corner.
    From the definition of $\Psi^{\mathrm{b}}$ one may check that for each concave corner, its vertical ray drawn on the front of the double and its horizontal ray on the back do not meet any edges of $\mathfrak{q}'$ (except perhaps at their endpoint).
    So we can reduce our task further to identifying at least one ray in each face $f$ of $\mathfrak{q}'$, since tracing the ray backwards will lead to a concave corner.
    Each face $f$ of $\mathbf{q}'$ is produced by a step $\mathfrak{R}^p_{\sigma}$ or $\mathfrak{L}^p_{\sigma}$ in the peeling exploration of $\mathbf{q}'$.
    Therefore $f$ is adjacent to one of the edges corresponding to a column of $\mathfrak{R}^p_{\sigma}$ or $\mathfrak{L}^p_{\sigma}$ that does not end on its base, shown in blue in Figure~\ref{fig:doubleRstep} and Figure~\ref{fig:doubleLstep}.
    But each of these also has a vertical ray of $\mathbf{r}'$ adjacent to it which is contained in $f$. 

    \ref{it:rightside} \& \ref{it:leftside}. Here we make convenient use of the mirror symmetry \eqref{eq:Psisymmetry} of the bijection: the mapping $\mathsf{Mirror}$ clearly interchanges right-tangential and left-tangential sides, while $\mathsf{Relabel}$ interchanges decreasing and increasing level lines.
    Moreover, assuming rigid quadrangulations are always drawn with the root corner on the bottom-right, at the same time $\mathsf{Mirror}$ interchanges vertical and horizontal tangential sides and $\mathsf{Relabel}$ interchanges odd and even level lines, where we say a level line is even/odd if the lowest adjacent label is even/odd.
    Because of this symmetry is suffices to check that horizontal (right- respectively left-)tangential sides are in correspondence with odd (decreasing respectively increasing) level lines.
    To this end we switch again to the expanded $\mathbf{r}'$ and unzipped $\mathfrak{q}'$.
    Horizontal tangential sides of $\mathbf{r}$ correspond to horizontal tangential sides of $\mathbf{r}'$ excluding the base.
    Recall from Section~\ref{sec:rowbyrow} that every step in the row-by-row exploration of $\mathbf{r}'$ reveals one horizontal side.
    As illustrated in Figure~\ref{fig:stepstangential}, if the type of exploration step is $\mathbf{G}_{\ell,p-\ell-1}$, $\mathbf{R}_\sigma^p$ or $\mathbf{L}_\sigma^p$ then this horizontal side is respectively non-tangential, right-tangential or left-tangential.
    \begin{figure}[h!]
        \centering
        \includegraphics[width=.75\linewidth]{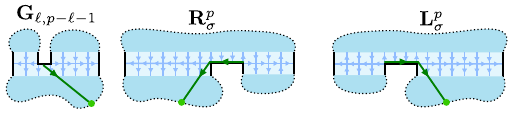}
        \caption{\label{fig:stepstangential}}
    \end{figure}
    The reason is that a shortest curve $\Gamma_x$ from a point $x$ on the horizontal side to the root corner has to cross the base of $\mathbf{G}_{\ell,p-\ell-1}$, $\mathbf{R}_\sigma^p$ or $\mathbf{L}_\sigma^p$ seen as a submap embedded in $\mathbf{r}'$. 
    On the side of colorful quadrangulations, the odd level lines are seen to be precisely the ones that traverse rings of quadrangles produced in the steps $\mathfrak{R}_\sigma^p$ or $\mathfrak{L}_\sigma^p$.
    Since they are decreasing for $\mathfrak{R}_\sigma^p$ and increasing for $\mathfrak{L}_\sigma^p$, this proves the correspondence.
    
    \ref{it:rightconvex} \& \ref{it:leftconvex}. Clearly a non-root convex corner of $\mathbf{r}$ is left/right-tangential if and only if it is adjacent to a left/right-tangential side.
    Again using symmetry it is sufficient to show that a non-root convex corner corresponds to an odd local minimum or an even local maximum in $\mathfrak{q}$ (or, equivalently, it is encircled by an odd level line) if and only if it is adjacent to a horizontal tangential side.
    This applies equally to the expanded rigid quadrangulation $\mathbf{r}'$ if we exclude the convex corners on the base and the unzipped colorful quadrangulation $\mathfrak{q}'$.
    By our discussion on \ref{it:rightside} \& \ref{it:leftside} above, a tangential horizontal side ending in a non-root convex corner is revealed in an exploration step $\mathbf{R}^p_{\sigma}$ and $\mathbf{L}^p_{\sigma}$ whenever $\sigma$ contains no $\downarrow$.
    The corresponding events $\mathfrak{R}_\sigma^p$ and $\mathfrak{L}_\sigma^p$ in the exploration of $\mathfrak{q}'$ reveal precisely the odd local minima or even local maxima. 
    Observing that the length of the horizontal side and the degree of the local maximum/minimum both agree with the length of the word $\sigma$, this finishes the proof of the claimed correspondences.
\end{proof}

\subsection{The $k$-fold base of a rigid quadrangulation via ascent paths}\label{sec:ascentpath}

The purpose of this subsection is to determine the image of $\mathcal{R}^{(\ell_1,\ldots,\ell_k)}$ under $\Psi$ and use this to prove Theorem~\ref{thm:basebijection}.
To state the main result of this subsection, we define for each $k \geq 1$ and $\ell_1,\ldots, \ell_k$ a unique colorful quadrangulation $\mathfrak{w}^{(\ell_1,\ldots,\ell_k)}$ with a single hole as follows.

For each $i=1,\ldots,k$, we take $\ell_i$ quadrangles of label $(i-1,i,i+1,i)$ and glue them sequentially as in Figure~\ref{fig:ascent}a: a face $f$ with label $(i-1,i,i+1,i)$ is glued to the next face $f'$ along the edge of $f$ just to the right of $i+1$ if $i+1$ is even or to the left of $i+1$ if $i+1$ is odd, noting that there is a unique edge of $f'$ that has compatible labels.
This gives a colorful quadrangulation with a hole (corresponding to the outer face) and the hole has a single corner of label $k+1$.
Finally, $\mathfrak{w}^{(\ell_1,\ldots,\ell_k)}$ is obtained from this by gluing together the two edges adjacent to corner of label $k+1$, see Figure~\ref{fig:ascent}b.
We root $\mathfrak{w}^{(\ell_1,\ldots,\ell_k)}$ on the unique oriented edge of the first quadrangle that points from $1$ to $0$ in clockwise order.
Moreover, the hole is marked at the edge that shares its endpoint with label $0$ with the root edge.

\begin{proposition}\label{prop:ascentsubmap}
    Let $\mathbf{r} \in \mathcal{R}$ and $\mathfrak{q} = \Psi(\mathbf{r})$. Then $\mathbf{r} \in \mathcal{R}^{(\ell_1,\ldots,\ell_k)}$ if and only if $\mathfrak{w}^{(\ell_1,\ldots,\ell_k)} \subset \mathfrak{q}$.
\end{proposition}

\noindent
The central idea is that on the double of a rigid quadrangulation $\mathbf{r} \in \mathcal{R}$ we can identify a natural path $\alpha(\mathbf{r})$, called the \emph{ascent path}, starting at the root and tracing the boundary of $\mathbf{r}$ in clockwise order until it reaches a convex corner.
More precisely, it starts travelling horizontally on the front just above the boundary (see Figure~\ref{fig:ascent}c).
Upon encountering a concave corner it circles around until hitting the vertical boundary, where it switches to the back and travels vertically a little away from the boundary. 
At the next concave corner it circles around again until hitting the horizontal boundary and switches to the front, and so on until ending at a convex vertex.

\begin{lemma}\label{lem:ascentcrossing}
    If $\mathbf{r}\in\mathcal{R}^{(\ell_1,\ldots,\ell_k)}$, the ascent path $\alpha(\mathbf{r})$ crosses $\ell_1$ edges of $\Psi(\mathbf{r})$ with label $(1,2)$, then $\ell_2$ edges with $(2,3)$, and so on until it finally crosses $\ell_k$ edges with label $(k,k+1)$. Moreover,
    \begin{enumerate}[label = (\roman*)] 
        \item\label{it:ascentendpoint} for each $i=1, \ldots, k$, all $\ell_i$ edges with label $(i,i+1)$ share the same endpoint with label $i+1$;
        \item\label{it:ascentevenodd} viewed along $\alpha(\mathbf{r})$ even labels are always on the left and odd labels on the right;
        \item\label{it:ascentlast} the last $\ell_k$ edges crossed by $\alpha(\mathbf{r})$ are all the edges incident to the degree-$\ell_k$ vertex at the endpoint of $\alpha(\mathbf{r})$.
    \end{enumerate}
\end{lemma}
\begin{figure}[h!]
    \centering
    \includegraphics[width=\linewidth]{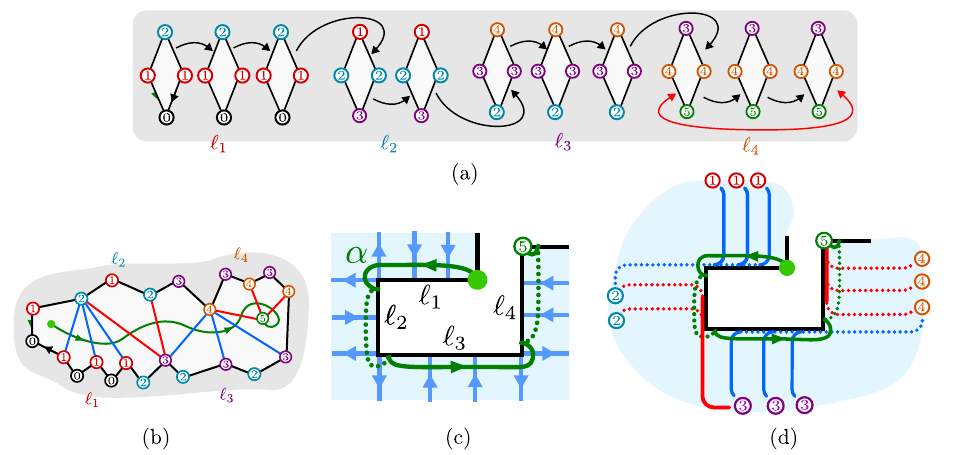}
    \caption{Illustrations in the case $(\ell_1,\ell_2,\ell_3,\ell_4) = (3,2,3,3)$. (a) Gluing of quadrangles in the construction of $\mathfrak{w}^{(\ell_1,\ldots,\ell_k)}$. The red arrow indicates the final gluing of the two edges on the hole adjacent to label $k+1$. (b) The colorful quadrangulation $\mathfrak{w}^{(\ell_1,\ldots,\ell_k)}$ with its hole on the outside. The marked edge of the hole is indicated by a green arrow. (c) The base of a rigid quadrangulation $\mathbf{r} \in \mathcal{R}^{(\ell_1,\ell_2,\ell_3,\ell_4)}$ with its ascent path $\alpha(\mathbf{r})$ in green. (d) An schematic illustration of the edges of $\Psi(\mathbf{r})$ associated to the rows (red) and columns (blue) ending on the base, that are each crossed exactly once by $\alpha(\mathbf{r})$. Dotted lines are drawn on the back of the double.\label{fig:ascent}}
\end{figure}
\begin{proof}
    These mentioned edges are precisely the edges associated to the rows and columns that start at the base.
    For $i=1,\ldots,k$, the $\ell_i$ rows or columns ending on the $i$th side of the base have transverse orientation in the direction of $\alpha(\mathbf{r})$ and therefore the corresponding edge has label $(i,i+1)$.
    Tracing these edges in the direction of the base, they turn in the direction parallel to $\alpha(\mathbf{r})$ and therefore merge on their way to a single vertex of label $i+1$ (item \ref{it:ascentendpoint}).
    We claim that $\alpha(\mathbf{r})$ crosses each of them once (and does so transversally), in the order in which they end on the base starting from the root.
    When $\alpha(\mathbf{r})$ crosses a column on the base, it is traveling on the front and therefore crosses the corresponding edge, which includes the vertical midsection of the column on the front, with the odd label on its right.
    In the way that $\alpha(\mathbf{r})$ circles around the next concave corner, it avoids further crossings with this edge (see Figure~\ref{fig:ascent}d).
    Similarly for a row, $\alpha(\mathbf{r})$ travels vertically on the back and crosses the horizontal midsection of the corresponding edge, this time with the even label on the right (note that the front and back have opposite parity in the planar drawing).
    The ascent path $\alpha(\mathbf{r})$ does not cross any other edges of the quadrangulation, because rows and columns that do not end on the base are transversally oriented away from it and therefore the corresponding edges do not touch the base, nor the ascent path.

    For statement \ref{it:ascentlast}, we note that the final stretch of  $\alpha(\mathbf{r})$ after its last edge crossing and the stretch of $\alpha(\mathbf{r})$ where it circles around the last concave corner (in case $k > 1$, or the initial stretch in case $k=0$) are in the same face of $\Psi(\mathbf{r})$.
    Indeed, one can travel between these stretches without crossing any edges by moving parallel to the boundary on the front if $k$ is even (like in Figure~\ref{fig:ascent}d) or on the back if $k$ is odd. 
\end{proof}

An important consequence of Lemma~\ref{lem:ascentcrossing} is that the path $\alpha(\mathbf{r})$ can be deduced from local information in the colorful quadrangulation $\mathfrak{q} = \Psi(\mathbf{r})$.
Informally, it can be seen as a unique path from the root to a local maximum that is weakly ascending with respect to the labels, explaining the terminology "ascent". 
To be precise, it starts in the center of the face to the right of the root edge of $\mathfrak{q}$.
Then iteratively, if $v$ is the vertex of maximal label in its present face, the path crosses the edge to the right of $v$ if $v$ has even label or to the left of $v$ if $v$ has odd label. 
It stops once the path has made a full turn around a vertex $v$, which is then necessarily a local maximum of $\mathfrak{q}$.

\begin{proof}[Proof of Proposition~\ref{prop:ascentsubmap}]
If $\mathbf{r}\in\mathcal{R}^{(\ell_1,\ldots,\ell_k)}$, then the unique submap $\mathfrak{w} \subset \mathfrak{q}$ whose inner edges (i.e.\ the edges that are not adjacent to a hole) are the edges crossed by the ascent path is seen to be precisely $\mathfrak{w}^{(\ell_1,\ldots,\ell_k)}$.
Vice versa, if $\mathfrak{w}^{(\ell_1,\ldots,\ell_k)} \subset \mathfrak{q}$, then by the description of the ascent path on $\mathfrak{q}$ that we just gave and the construction of $\mathfrak{w}^{(\ell_1,\ldots,\ell_k)}$, the ascent path $\alpha(\mathbf{r})$ crosses all inner edges of $\mathfrak{w}^{(\ell_1,\ldots,\ell_k)}$ and ends on its unique inner vertex.
Judging by the labels of these edges, we have $\mathbf{r} \in \mathcal{R}^{(\ell_1,\ldots,\ell_2)}$.
\end{proof}

\noindent
We are now ready to prove Theorem~\ref{thm:basebijection}.

\begin{proof}[Proof of Theorem~\ref{thm:basebijection}]
Proposition~\ref{prop:ascentsubmap} says that $\Psi$ restricts to a bijection between $\mathcal{R}^{(\ell_1,\ldots,\ell_k)}$ and $\{ \mathfrak{q} \in \mathcal{C} :  \mathfrak{w}^{(\ell_1,\ldots,\ell_k)} \subset \mathfrak{q}\}$.
A colorful quadrangulation $\mathfrak{q} = \Psi(\mathbf{r})$ in the latter set is bijectively determined by the colorful quadrangulation $\mathfrak{u}$ that has to be glued in the hole of $\mathfrak{w}^{(\ell_1,\ldots,\ell_k)}$ to recover $\mathfrak{q}$.
Recall that in this gluing procedure the root of $\mathfrak{u}$ is glued to the marked edge on the hole.
The only condition on a colorful quadrangulation of the disk to appear as $\mathfrak{u}$ in this way is that the labels along its boundary match with the labels on the hole of $\mathfrak{w}^{(\ell_1,\ldots,\ell_k)}$, which is precisely the condition that $\mathfrak{u} \in \mathcal{C}^{(\ell_1,\ldots,\ell_k)}$.

Moreover, $\mathfrak{u}$ is obtained from $\mathfrak{q}$ by deleting all inner edges and inner vertices, but these are precisely the edges corresponding to rows and columns that end on the base of $\mathbf{r}$ and the vertex corresponding to the endpoint of the base.
In other words, $\mathfrak{u} = \Psi^{\mathrm{b}}(\mathbf{r})$.
This shows that $\Psi^{\mathrm{b}}$ restricts to a bijection between $\mathcal{R}^{(\ell_1,\ldots,\ell_k)}$ and $\mathcal{C}^{(\ell_1,\ldots,\ell_k)}$.
\end{proof}

\subsection{Restriction to type B and C}

\begin{proof}[Proof of Corollary~\ref{cor:doublebase}]
Recall from the introduction that a base-$p$ rigid quadrangulation $\mathbf{r}\in \mathcal{R}^{(p)}$ is of type B if the vertical side to the right of the base of $\mathbf{r}$ ends in a convex corner $c$, and that $\mathcal{R}^{(p)(q)}_{\mathrm{B}}$ is the set of such rigid quadrangulation for which this vertical side is of length $q$.
In the row-by-row exploration $\mathbf{e}_0\subset \mathbf{e}_1 \subset \cdots \subset \mathbf{r}$, this condition amounts precisely to the first $q-1$ steps all being of type $\mathbf{G}_{k,\ell}$ for some $k \geq 0$ and $\ell \geq 1$ or type $\mathbf{L}_\sigma^k$ and the $q$th step being $\mathbf{G}_{k,0}$ for some $k\geq 0$, because $\mathbf{G}_{k,0}$ is the only step that reveals a top-right convex corner.  

On the other hand, a colorful quadrangulation $\mathfrak{q}\in \mathcal{C}^{(p)}$ is of type B if the endpoint $v_0$ of the root edge of $\mathfrak{q}$ with label $0$ is a local minimum, and $\mathfrak{q}$ belongs to $\mathcal{C}^{(p)(q)}_{\mathrm{B}}$ if this vertex has degree $q$.
In terms of the peeling exploration $\mathfrak{e}_0 \subset \mathfrak{e}_1 \subset \cdots \subset \mathfrak{q}$, this corresponds to the first $q-1$ steps of type $\mathfrak{G}_{k,\ell}$ for some $k \geq 0$ and $\ell \geq 1$ or type $\mathfrak{L}_\sigma^k$ and the $q$th step being $\mathfrak{G}_{k,0}$ for some $k\geq 0$.
Indeed, while $v_0$ is still on the boundary of the (first) hole of $\mathfrak{e}_i$ a step of type $\mathfrak{R}_\sigma^k$ would reveal a $(-1,0)$-edge incident to $v_0$.
Without such steps the first time in the exploration that vertex $v_0$ is not adjacent to the hole anymore is after the first occurrence of a step $\mathfrak{G}_{k,0}$ for some $k \geq 0$, because that glues the two edges adjacent to $v_0$ on the hole together.
Every step until that time reveals precisely one new $(0,1)$-edge incident to $v_0$, implying that $v_0$ is fully explored after time $q$.

According to Theorem~\ref{thm:bijection}, the described conditions are equivalent in case $\mathfrak{q} = \Psi^{(p)}(\mathbf{r})$.
Proposition~\ref{prop:bijectionequivalence} then implies that for every $p,q \geq 1$ the mapping $\Psi^{\mathrm{b}}$ restricts to a bijection between $\mathcal{R}^{(p)(q)}_{\mathrm{B}}$ and $\mathcal{C}^{(p)(q)}_{\mathrm{B}}$.

To verify that it further restricts to a bijection between $\mathcal{R}^{(p)(q)}_{\mathrm{C}}$ and $\mathcal{C}^{(p)(q)}_{\mathrm{C}}$, we need to show that for $\mathbf{r} \in \mathcal{R}^{(p)(q)}_{\mathrm{B}}$ the endpoint $v_0$ of the root edge of $\Psi^{\mathrm{b}}(\mathbf{r})$ is incident only once to the root face if and only if $\mathbf{r}$ contains a rectangle of size $p \times q$ touching the root corner.
Recalling the notation of $\vec{v}_i(\mathfrak{q})$ from Proposition~\ref{prop:extendbijequivalence}, the former condition is equivalent to requiring that $\vec{v}_i(\mathfrak{q}) \neq v_0$ for $i = 2, \ldots, p$.
So on $\mathbf{r}$ this is equivalent to none of the paths $\vec{\gamma}_2 (\mathbf{r}), \cdots, \vec{\gamma}_p (\mathbf{r})$ ending on $c$.
The illustration in Figure~\ref{fig:relationCtype} should make clear that is indeed equivalent to $\mathbf{r}$ containing a rectangle of size $p \times q$.
\end{proof}
\begin{figure}[h!]
    \centering
    \includegraphics[width=.5\linewidth]{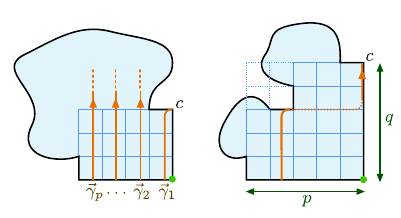}
    \caption{If $\mathbf{r}$ contains the full rectangle of size $p\times q$ (left), then none of $\vec{\gamma}_2 (\mathbf{r}), \cdots, \vec{\gamma}_p (\mathbf{r})$ end on $c$ because they exit the rectangle upwards. Otherwise (right) at least one of them is deflected to the right and ends on $c$.\label{fig:relationCtype}}
\end{figure}

\section{Enumeration}

\subsection{Generating functions}

With the bijections in place it is a straightforward task to translate the enumeration results of Bousquet-M\'elou and Elvey Price in \cite{Bousquet-Melou_generating_2020,Bousquet-Melou_Refined_2025} for colorful quadrangulations to rigid quadrangulations.

\begin{proof}[Proof of Theorem~\ref{thm:rigidquadenum}]
    According to \cite[Theorem~1.2 \& Corollary~5.2]{Bousquet-Melou_generating_2020} the generating function of colorful labeled quadrangulations with weight $t$ per face is $\frac{1}{2t^2}(t- 2t^2 - R(t))$. 
    Exactly half of these have labels $(0,1,2,1)$ on the face to the right of the root.
    Since quadrangulations have two more vertices than faces, we find that the generating function for $\mathcal{C}$ with weight $t$ per vertex is $\frac{1}{4}(t- 2t^2 - R(t))$.
    By Theorem~\ref{thm:rootedbijection} this is also the generating function of $\mathcal{R}$.
    The asymptotic enumeration is taken from \cite[Theorem~1.2]{Bousquet-Melou_generating_2020}, where $|\mathcal{R}_n| = g_{n-2}$.
\end{proof}

\begin{proof}[Proof of Corollary~\ref{cor:multivariate}]
    The generating functions for $\mathcal{R}^{(p)}$, $\mathcal{R}^{(p)(q)}_{\mathrm{C}}$ and $\mathcal{R}^{(p)(q)}_{\mathrm{B}}$ are direct translations of those appearing in \cite[Theorem~7.1]{Bousquet-Melou_generating_2020}.
    To be precise, according to \cite[Definition~3.2 and equation (15)]{Bousquet-Melou_generating_2020}, $\mathsf{P}(t,y)$ enumerates P-patches with weight $t$ per inner face and $y$ for the half-perimeter while $\mathcal{P}(t,y) = t \mathsf{P}(t,ty)$ counts these with a weight $t$ per vertex.   
    Excluding the ``atomic patch'' with zero perimeter and a single vertex, the generating function of $\mathcal{C}^{(p)}$ is thus $P(t,y) = \mathcal{P}(t,y) - t$.
    By Theorem~\ref{thm:basebijection} it is also the generating function of rigid quadrangulations $\mathcal{R}^{(p)}$ with weight $t$ per convex corner that is not on the base, of which there are $n(\mathbf{r})-1$.
    
    Similarly, their series $\mathcal{C}(t,x,y) = \mathsf{C}(t,x,ty)$ enumerates C-patches with weight $t$ per non-root vertex, $y$ for the half-perimeter and $x$ for the degree of the root vertex. 
    These are precisely the C-type colorful quadrangulations $\mathcal{C}^{(p)(q)}_{\mathrm{C}}$ (after reversal of the root edge).
    By Corollary~\ref{cor:doublebase}, $C(t,x,y)=\mathcal{C}(t,x,y)$ also enumerates C-type rigid quadrangulations $\mathcal{R}^{(p)(q)}_{\mathrm{C}}$ with weight $t$ per convex corner, excluding the root corner as well as the two adjacent convex corners.
    According to \cite[Lemma 3.4]{Bousquet-Melou_generating_2020}, the generating function of B-patches, including the atomic patch, is $\mathcal{B}(t,x,y) = 1 / (1- \mathcal{C}(t,x,y))$.
    Then $B(t,x,y) = \mathcal{B}(t,x,y) -1$ is the generating function of B-type colorful quadrangulations $\mathcal{C}^{(p,q)}_{\mathrm{B}}$ and therefore also of B-type rigid quadrangulations $\mathcal{R}^{(p)(q)}_{\mathrm{B}}$.
    This proves the relation $C(t,x,y) = 1 - 1 / (1+B(t,x,y))$.

    Finally, E-patches are defined in \cite[Definition~2.8]{Bousquet-Melou_Refined_2025} to be colorful quadrangulations of the disk with labels around the boundary of the type $0,1,0,1,\ldots,0,-1,0,-1,\ldots,0,-1$.
    The generating function for E-patches is identified in \cite[Proposition~3.9]{Bousquet-Melou_Refined_2025}, setting $v=1$ and $\omega=0$. Comparing to \cite[Lemma~6.4]{Bousquet-Melou_generating_2020} shows that 
    \begin{align}
    \mathsf{E}(t,x,y) = \frac{1}{x y} \mathsf{C}\left(t,\frac{x}{t},y\right) = \frac{1}{xy}C\left(t,\frac{x}{t},\frac{y}{t}\right)    
    \end{align}
    enumerates E-patches with weight $t$ per inner face, weight $y$ per label $1$ and weight $x$ per label $-1$ on the boundary.
    Then 
    \begin{align}
        E(t,x,y) &= \mathsf{E}(t,t x,t y) - \mathsf{E}(t,0,t y) - \mathsf{E}(t,t x,0) + \mathsf{E}(t,0,0) \nonumber\\
        &= \frac{1}{xy} C(t,x,y) - P(t,x) - P(t,y) - t
    \end{align}
    enumerate E-patches with weight $t$ per non-root vertex and labels $1$ and $-1$ appearing both at least once on the boundary.
    In the second line we used the identity $[x^1]\mathsf{C}(t,x,y) = y \mathsf{P}(t,y)$ from \cite[Theorem~6.1]{Bousquet-Melou_generating_2020}. 
    By shifting all labels by $+1$ these are precisely the colorful quadrangulations $\mathcal{C}^{(p,q)}$.
    So by Theorem~\ref{thm:basebijection} it is also the generating function of $\mathcal{R}^{(p,q)}$ weighted by $t^{n(\mathbf{r})-2}$.    
\end{proof}

\subsection{Explanation of the catalytic equations}\label{sec:catalyticexplanation}

Bousquet-M\'elou and Elvey Price found these generating functions by showing that they uniquely solve a (2-catalytic) system of equations resulting from bijective decompositions of colorful quadrangulations, see \cite[Theorem~6.1]{Bousquet-Melou_generating_2020}.
In our notation these equations boil down to
\begin{align}
    P(t,y) &= [x^0]\frac{C(t,x,y)-t x y}{x y},\label{eq:PCrelation}\\
    B(t,x,y) &= \frac{1}{1-C(t,x,y)} - 1, \label{eq:BCrelation}\\
    C(t,x,y) &= xy \left(E(t,x,y) + P(t,x)+P(t,y)+t\right), \label{eq:BErelation} \\
    E(t,x,y) &= [x^{>0}]P(t,x)B\left(t,\frac{1}{1-x^{-1}},y\right),\label{eq:EPBrelation}
\end{align}
where in the first line $[x^0]$ means extracting the coefficient of $x^0$ in the series, and in the last line $B\left(t,\frac{1}{1-x^{-1}},y\right)$ is regarded as a formal power series in $x^{-1}$ while $[x^{>0}]$ means keeping only the terms with positive powers of $x$ in the resulting Laurent series.  
Let's show that they have simple interpretations at the level of rigid quadrangulations.

\begin{figure}[h]
    \centering
    \includegraphics[width=\linewidth]{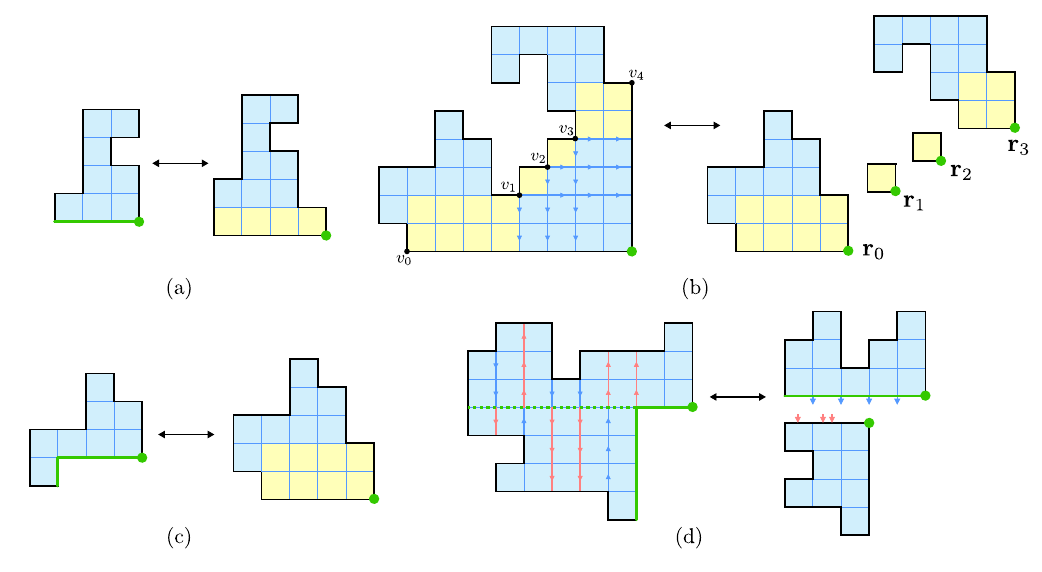}
    \caption{Illustrating the bijective relations underlying equations \eqref{eq:PCrelation}, \eqref{eq:BCrelation}, \eqref{eq:BErelation} and \eqref{eq:EPBrelation}.\label{fig:rigiddecomp}}
\end{figure}

The first equation \eqref{eq:PCrelation} reflects that for $p\geq 1$, gluing a row of length $p+1$ to the base of a rigid quadrangulation in $\mathcal{R}^{(p)}$ yields a C-type rigid quadrangulation in $\mathcal{R}^{(p+1)(1)}_{\mathrm{C}}$ with one convex corner extra, see Figure~\ref{fig:rigiddecomp}a.
The only rigid quadrangulation in $\mathcal{R}^{(p+1)(1)}_{\mathrm{C}}$ not obtained in this way is $\square$ for $p=0$ with weight $t x y$.

For the second equation \eqref{eq:BCrelation}, recall that a B-type $\mathbf{r} \in \mathcal{R}^{(p)(q)}_{\mathrm{B}}$ is also a C-type if $\mathbf{r}$ contains a rectangle of size $p\times q$ with corner at the root. 
This is equivalent to the absence of concave corners for which both rays end on a side adjacent to the root. 
This interpretation leads to a natural decomposition of a B-type rigid quadrangulation $\mathbf{r}\in\mathcal{R}^{(p)(q)}_{\mathrm{B}}$, see Figure~\ref{fig:rigiddecomp}b: let $v_1,\ldots,v_k$ be all $k\geq 0$ concave corners for which both rays end on a side adjacent to the root of $\mathbf{r}$, ordered from left to right.
Let further $v_0$ and $v_{k+1}$ be the endpoints of the horizontal and vertical side of $\mathbf{r}$ adjacent to the root.
Then for each $i=0,\ldots,k$, the rectangle with corners $v_i$ and $v_{i+1}$ is fully contained in $\mathbf{r}$ and cutting out this rectangle together with everything that connects to its left or top side gives a C-type rigid quadrangulation $\mathbf{r}_i$.
It is not hard to see that $\mathbf{r} \to (\mathbf{r}_0,\ldots,\mathbf{r}_k)$ is indeed a bijective correspondence between B-type and non-empty sequences of C-type rigid quadrangulations, such that rows, columns that end on a side adjacent to the root and the convex corners that are not on a side adjacent to the root nicely partition.
Hence, $B(t,x,y) = \sum_{k\geq 0} C(t,x,y)^{k+1}$.

For the third equation \eqref{eq:BErelation}, we note that for each $p,q\geq 1$ gluing a rectangle of size $(p+1)\times (q+1)$ to a rigid quadrangulation with 2-fold base $(\ell_1,\ell_2) = (p,q)$ in $\mathcal{R}^{(p,q)}$ gives a C-type rigid quadrangulation in $\mathcal{R}^{(p+1)(q+1)}_{\mathrm{C}}$, see Figure~\ref{fig:rigiddecomp}c. 
Then we only need to add to these the C-type rigid quadrangulations $\mathcal{R}^{(p+1)(q+1)}_{\mathrm{C}}$ with $p=0$ or $q=0$, for which we can refer back to \eqref{eq:PCrelation}.

The final equation \eqref{eq:BErelation} can be understood as follows.
Given a rigid quadrangulation with 2-fold base in $\mathbf{r}\in \mathcal{R}^{(p,q)}$, we can cut it into two pieces along the ray that extends the first side of the base (dotted green line in Figure~\ref{fig:rigiddecomp}d). 
Next, we get rid of all rays that point away from the new portion of the boundary of these pieces by merging their neighboring columns.
Then the top piece $\mathbf{r}_1$ is exactly of the type of $\mathcal{R}^{(\ell)}$ for some $\ell \geq p$ and the bottom piece $\mathbf{r}_2$ of B-type $\mathcal{R}^{(q)(m)}_{\mathrm{B}}$ for some $m \geq 1$.
In order for this operation to be bijective we need to record for each of the $m$ boundary edges of $\mathbf{r}_1$ how many downward rays were crossing it in $\mathbf{r}$.
The generating function for the bottom piece together with this information and weight $x^{-1}$ per downward ray is precisely $B\left(t,\frac{1}{1-x^{-1}},y\right)$.
The restriction to positive powers of $x$ in \eqref{eq:BErelation} is to take care of the inequality $\ell \geq p$. 

Besides the 2-catalytic system of equations, Bousquet-M\'elou and Elvey Price also identified a 1-catalytic equation in \cite[Proposition~9.3]{Bousquet-Melou_Refined_2025} that uniquely determines the generating function $P(t,y)$ of patches.
Denoting by $P^{(p)}(t) = [y^p]P(t,y) = \sum_{\mathbf{r}\in \mathcal{R}^{(p)}} t^{n(\mathbf{r})-1}$ the generating function of rigid quadrangulations with (1-fold) base of length $p$ for $p\geq 1$ and $P^{(0)}(t) = t$ by convention, it amounts to the system
\begin{align}
    P^{(p)}(t) = \sum_{\ell=0}^{p-1} P^{(\ell)}(t) P^{(p-\ell-1)}(t) + 2 \sum_{k,\ell\geq 0} \binom{\ell+k}{\ell} P^{(p+\ell)}(t) P^{(k)}(t) \qquad\text{for }p\geq 1.
\end{align}
Comparing with Lemma~\ref{lem:minimalsubmaps}, we see that the right-hand side reflects precisely the partition of $\mathcal{R}^{(p)}$ according to the type of the submap found at the first step of its row-by-row exploration.
Indeed, the summand in the first term on the right-hand side enumerates the rigid quadrangulations $\mathbf{r}$ containing $\mathbf{G}_{\ell,p-\ell-1}$. 
The summand of the second term accounts for the $\binom{\ell+k}{\ell}$ possible words $\sigma$ with $k$ letters $\downarrow$ and $\ell$ letters $\uparrow$ indexing the submaps $\mathbf{R}^p_{\sigma}$ and $\mathbf{L}^p_\sigma$.

\subsection{Rigid half-cylinders}\label{sec:trumpets}

With the decompositions of $B$-type and $C$-type rigid quadrangulations in place, the enumeration rigid half-cylinders is a straightforward matter. 
Recall the definition of the set $\mathcal{H}^{(p,q)}$ of rigid half-cylinders from Section~\ref{sec:jt}.

\begin{figure}[h]
    \centering
    \includegraphics[width=\linewidth]{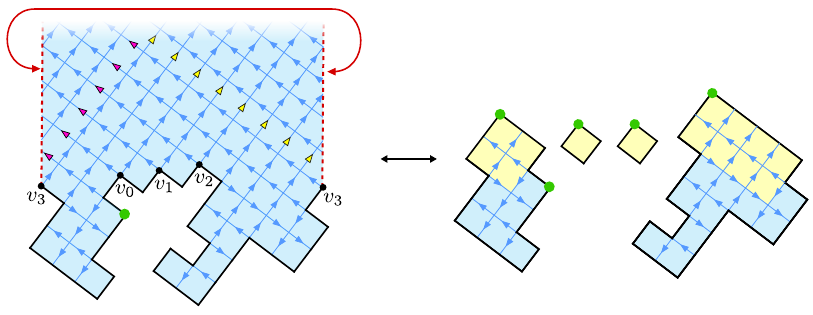}
    \caption{A rigid half-cylinder decomposes bijectively into a sequence of C-type rigid quadrangulation with the first one carrying an extra marked convex corner.\label{fig:trumpetdecomp}}
\end{figure}

\begin{proof}[Proof of Corollary~\ref{cor:trumpet}]
    The decomposition we use is very similar to that of a B-type rigid quadrangulations into a sequence of C-type quadrangulations (Figure~\ref{fig:rigiddecomp}b). 
    Let $v_0, \ldots, v_k$ be the concave corners of $\mathbf{h}\in \mathcal{H}^{(p,q)}$ at which two infinite rays start, in the order encountered when tracing the boundary of $\mathbf{h}$ starting at the root corner keeping the interior on the left, see Figure~\ref{fig:trumpetdecomp}.
    It should be clear that there is at least one such corner, so $k \geq 0$.
    Dissecting $\mathbf{h}$ along the rays starting at $v_0, \ldots, v_k$ yields a sequence of $k+1$ type-$C$ rigid quadrangulations rooted where the rays  intersected, and the first rigid quadrangulation in the sequence (the one between $v_k$ and $v_0$) carries an additional marked convex corner that does not share a side with the root corner.
    The total of lengths of the horizontal respectively vertical sides adjacent to the root (after rotating to the canonical orientation aligned with the upper-left quadrant) is $p$ respectively $q$.
    This is readily seen to be a bijective decomposition.
    At the level of formal generating functions it implies that
    \begin{align}
        {H^{(p,q)}}{}'(t) = [x^p y^q]\frac{\partial C(t,x,y)}{\partial t} \sum_{k=0}^\infty C(t,x,y)^k = [x^p y^q] \frac{\partial}{\partial t}\log \frac{1}{1-C(t,x,y)}.
    \end{align}
    Using $H^{(p,q)}(0) = 0$ and the formula for $C(t,x,y)$ from Corollary~\ref{cor:multivariate}, we thus find
    \begin{align*}
        H^{(p,q)}(t) &= [x^p y^q]\log\frac{1}{1-C(t,x,y)} = [x^p y^q]\log( B(t,x,y)+1) \\&= \sum_{n\geq \max(p,q)} \frac{1}{n+1}\binom{2n-p+1}{n}\binom{2n-q+1}{n} R(t)^{n+1}.
    \end{align*}
    The claimed expression for ${H^{(p,q)}}{}'(t)$ follows from computing the derivative.
\end{proof}

It remains to verify the claims concerning the asymptotic enumeration of the rigid trumpets.

\begin{proof}[Proof of Proposition~\ref{prop:trumpetasymp}]
Recall from Theorem~\ref{thm:rigidquadenum} that $R(t)$ is the formal power series solution to
\begin{align}\label{eq:Rdefrecap}
    \sum_{n=0}^\infty \frac{1}{n+1} \binom{2n}{n}^2 R(t)^{n+1} = t.
\end{align}
As shown in \cite[Proposition~8.3]{Bousquet-Melou_generating_2020} it has radius of convergence $t_*=\frac{1}{4\pi}$ and 
\begin{align}\label{eq:Rasymp}
    R(t) - \frac{1}{16} \sim \frac{1}{4} \frac{1-4\pi t}{\log(\frac{1}{1-4\pi t})}\qquad\text{as }t\to t_*.
\end{align}
By taking a derivative of \eqref{eq:Rdefrecap}, we have the identity
\begin{align}
    \frac{1}{R'(t)} = \sum_{n=0}^\infty \binom{2n}{n}^2 R(t)^n.
\end{align}
Therefore the rooted half-cylinder generating function $H^{(p,q)}{}'(t)$ of Corollary~\ref{cor:trumpet} can be expressed as
\begin{align}
    H^{(p,q)}{}'(t) = \frac{\sum_{n\geq \max(q,p)} \binom{2n-p+1}{n}\binom{2n-q+1}{n}R(t)^n}{\sum_{n\geq 0}\binom{2n}{n}^2 R(t)^n},
\end{align}
from which it is clear that $H^{(p,q)}{}'(t) \leq 1$ for $t < t_*$.
Hence, $H^{(p,q)}(t)$ has radius of convergence equal to $t_*$ and $H^{(p,q)}(t_*) < t_* H^{(p,q)}{}'(t_*)$ is finite.
Therefore the critical Boltzmann $(p,q)$-rigid half-cylinder at $t=t_*$ is well-defined.  

We now focus on the critical Boltzmann $(p,q)$-rigid half-cylinder $\mathbf{H}$ in the limit where $p$ and/or $q$ becomes large.
To lighten notation in the remainder of the proof we denote by $r = \sqrt{p^2 + q^2}$ the circumference of the cylinder, leaving the dependence on $p$ and $q$ implicit.
Writing for $\mu > 0$ fixed, $\mu_{p,q} \coloneqq \mu\, \frac{1}{r^2} \log r^2$, we are interested in estimating the expectation 
\begin{align}
    \mathbb{E}^{(p,q)}_{t_*}\left[ e^{-\mu_{p,q} n(\mathbf{H})}\right] &= \frac{1}{H^{(p,q)}(t_*)}\sum_{n \geq \max(p,q)} \frac{1}{n+1}\binom{2n-p+1}{n}\binom{2n-q+1}{n} R(t_* e^{-\mu_{p,q}})^{n+1}\label{eq:halfcylinderexpec}
\end{align}
as $r \to \infty$.
The asymptotic \eqref{eq:Rasymp} translates into
\begin{align}
    R(t_* e^{-\mu_{p,q}}) &= \frac{1}{16} \exp\left(-4 \frac{\mu_{p,q} + o(\mu_{p,q})}{\log(1/\mu_{p,q})}\right) = \frac{1}{16} \exp\left(-4 \frac{\mu}{r^2} + o(r^{-2})\right)\qquad \text{as } r\to\infty.
\end{align}
Approximating the binomials by a normal distribution, one may verify that there exists a constant $C>0$ such that for all $n,p,q \geq 1$ we can bound
\begin{align}
    \frac{1}{n+1}\binom{2n-p+1}{n}\binom{2n-q+1}{n} 16^{-n-1} \leq C \,\frac{2^{-p-q}}{n^2} \exp\left(- \frac{r^2}{4n}\right),
\end{align}
while we have the asymptotic estimate 
\begin{align}
    \frac{1}{n+1}\binom{2n-p+1}{n}\binom{2n-q+1}{n} 16^{-n-1} = \frac{2^{-p-q-2}}{\pi n^2} \exp\left(- \frac{r^2}{4n} + O(n^{-1/2})\right)\qquad \text{as }n\to\infty,
\end{align}
where the error $O(n^{-1/2})$ is uniform in $p$ and $q$ for, say, $r \leq  n^{2/3}$.
The sum in \eqref{eq:halfcylinderexpec} restricted to the range $n \leq r^{3/2}$ is bounded above by
\begin{align}
    \sum_{ n \leq r^{3/2}} &\frac{1}{n+1}\binom{2n-p+1}{n}\binom{2n-q+1}{n} 16^{-n-1}  \leq  C\,2^{-p-q}\sum_{n \leq r^{3/2}} \frac{1}{n^2} \exp\left(- \frac{r^2}{4n} \right) \nonumber\\
    & \leq C \, 2^{-p-q} \int_0^{r^{3/2}} \frac{\rmd x}{x^2} \exp\left(-\frac{r^2}{4x}\right) = 4 C\, \frac{2^{-p-q}}{r^2}e^{-\tfrac{1}{4}\sqrt{r}}
\end{align}
for all $p,q \geq 1$. 
The sum restricted to $n \geq r^{3/2}$ evaluates asymptotically to
\begin{align}
    \sum_{n \geq r^{3/2}} \frac{1}{n+1}&\binom{2n-p+1}{n}\binom{2n-q+1}{n} 16^{-n-1}  \exp\left( -4(\mu + o(1)) \frac{n}{r^2}\right) \nonumber\\
    &= \frac{2^{-p-q-2}}{\pi}\sum_{n \geq r^{3/2}} \frac{1}{n^2} \exp\left(- \frac{r^2}{4n} -(\mu+o(1)) \frac{4n}{r^2}+O(n^{-1/2})\right)\nonumber\\
    &= \frac{2^{-p-q-2}}{\pi} (1+o(1)) \int_{r^{3/2}}^\infty \frac{\rmd x}{x^2} \exp\left(-\frac{r^2}{4x}-(\mu+o(1)) \frac{4x}{r^2}\right) \nonumber\\
    &= \frac{2^{-p-q}}{\pi r^2} (1+o(1))\int_{4 r^{-1/2}}^\infty \frac{\rmd t}{t^2} \exp\left(-\frac{1}{t}-(\mu+o(1))t\right)\nonumber\\
    &= \frac{2^{-p-q}}{\pi r^2} \left(\sqrt{4\mu}\, K_1\left(\sqrt{4\mu}\right) + o(1) \right)  \qquad\text{as } r\to\infty,
\end{align}
where $K_1$ is a modified Bessel function of the second kind.
In particular, since $\sqrt{4\mu} K_1(\sqrt{4\mu}) \to 1$ as $\mu \to 0$ this implies
\begin{align}
    H^{(p,q)}(t_*) = \sum_{n\geq \max(p,q)} \frac{1}{n+1} \binom{2n-p+1}{n}\binom{2n-q+1}{n} R(t_*)^{n+1} \sim \frac{2^{-p-q}}{\pi r^2}\quad \text{as }r\to\infty.
\end{align} 
Furthermore,
\begin{align}
    \mathbb{E}^{(p,q)}_{t_*}\left[ e^{-\mu_{p,q} n(\mathbf{T})}\right] &\to \sqrt{4\mu}\, K_1\left(\sqrt{4\mu}\right) = \int_0^\infty \rmd t \,e^{ -t - \frac{\mu}{t}}\qquad \text{as }r\to\infty.
\end{align} 
This implies the convergence in distribution of $n(\mathbf{R}) / (r^2/\log r^2)$ to an inverse-exponential $1/\mathcal{E}(1)$.
Since $\log n(\mathbf{R}) / \log r^2$ converges to $1$ in probability, we have the same convergence in distribution for $(n(\mathbf{R}) \log n(\mathbf{R})) / r^2$.
\end{proof}

Let us finish by explaining how this distributional convergence is equivalent to the convergence of unnormalized measures quoted in Section~\ref{sec:jt}.
Recall the measure $\mu^{(p,q)}$ that assigns mass $\mu^{(p,q)}(\{\mathbf{h}\}) = 2^{p+q} t_*^{n(\mathbf{h})+1} / (n(\mathbf{h})+1)$ to a rigid half-cylinder $\mathbf{h}$.
By definition of the generating function $H^{(p,q)}(t)$ (see Corollary~\ref{cor:trumpet}) this measure has total mass
\begin{align}
    \mu^{(p,q)}\left(\mathcal{H}^{(p,q)}\right) = 2^{p+q} H^{(p,q)}(t_*).
\end{align}
According to Proposition~\ref{prop:trumpetasymp} it is asymptotically equivalent to $\pi^{-1}(p^2+q^2)^{-1}$ as $p^2+q^2 \to \infty$.
Recalling also the definition $\mathsf{Len}^\varepsilon = \frac{1}{2}\varepsilon^2 n(\mathbf{h}) \log n(\mathbf{h})$ of the renormalized boundary length, we observe that for fixed $b>0$ and $\varepsilon = b / \sqrt{p^2+q^2}$,
\begin{align*}
    \pi (p^2+q^2)\, \mathsf{Len}^\varepsilon_* \mu^{(p,q)} = \frac{\pi b^2}{\varepsilon^2}\, \mathsf{Len}^\varepsilon_* \mu^{(p,q)}
\end{align*}
is precisely the probability distribution under $\mathbb{P}^{(p,q)}_{t_*}$ of the random variable
\begin{align*}
    \mathsf{Len}^\varepsilon(\mathbf{H}) = \frac{b^2}{2} \frac{n(\mathbf{H})\log n(\mathbf{H})}{p^2+q^2}.
\end{align*}
By Proposition~\ref{prop:trumpetasymp}, the latter converges in distribution to $\frac{b^2}{2} \frac{1}{\mathcal{E}(1)}$ and \eqref{eq:Zhalfcyl} follows from rearranging constant factors.

\bibliographystyle{siam}
\bibliography{diskimmersion}	

\begin{thebibliography}{10}

\bibitem{Aidekon_Growth_2022}
{\sc E.~A\"{\i}d\'{e}kon and W.~Da~Silva}, {\em Growth-fragmentation process embedded in a planar {B}rownian excursion}, Probab. Theory Related Fields, 183 (2022), pp.~125--166.

\bibitem{Aidekon_scaling_2024}
{\sc {\'E}.~A{\"\i}d{\'e}kon, W.~Da~Silva, and X.~Hu}, {\em The scaling limit of the volume of loop {O(n)} quadrangulations}, arXiv preprint arXiv:2402.04827,  (2024).

\bibitem{Aru_Mating_2021}
{\sc J.~Aru, N.~Holden, E.~Powell, and X.~Sun}, {\em Mating of trees for critical {Liouville} quantum gravity}, arXiv preprint arXiv:2109.00275,  (2021).

\bibitem{Aru_Brownian_2023}
\leavevmode\vrule height 2pt depth -1.6pt width 23pt, {\em Brownian half-plane excursion and critical {L}iouville quantum gravity}, J. Lond. Math. Soc. (2), 107 (2023), pp.~441--509.

\bibitem{Bernardi_Bijective_2007}
{\sc O.~Bernardi}, {\em Bijective counting of {K}reweras walks and loopless triangulations}, J. Combin. Theory Ser. A, 114 (2007), pp.~931--956.

\bibitem{Bertoin_Martingales_2018}
{\sc J.~Bertoin, T.~Budd, N.~Curien, and I.~Kortchemski}, {\em Martingales in self-similar growth-fragmentations and their connections with random planar maps}, Probab. Theory Related Fields, 172 (2018), pp.~663--724.

\bibitem{Bertoin_Self_2024}
{\sc J.~Bertoin, N.~Curien, and A.~Riera}, {\em Self-similar markov trees and scaling limits}, arXiv preprint arXiv:2407.07888,  (2024).

\bibitem{Borot_recursive_2012}
{\sc G.~Borot, J.~Bouttier, and E.~Guitter}, {\em A recursive approach to the {$O(n)$} model on random maps via nested loops}, J. Phys. A, 45 (2012), pp.~045002, 38.

\bibitem{Bousquet-Melou_generating_2020}
{\sc M.~Bousquet-M\'{e}lou and A.~Elvey~Price}, {\em The generating function of planar {E}ulerian orientations}, J. Combin. Theory Ser. A, 172 (2020), pp.~105183, 48.

\bibitem{Bousquet-Melou_Eulerian_2020}
{\sc M.~Bousquet-M\'{e}lou, A.~Elvey~Price, and P.~Zinn-Justin}, {\em Eulerian orientations and the six-vertex model on planar maps}, S\'{e}m. Lothar. Combin., 82B (2020), pp.~Art. 70, 12.

\bibitem{Bousquet-Melou_Refined_2025}
{\sc M.~Bousquet-Mélou and A.~E. Price}, {\em Refined enumeration of planar {Eulerian} orientations},  (2025).
\newblock arXiv preprint arXiv:2503.15046.

\bibitem{Budd_peeling_2018}
{\sc T.~Budd}, {\em The peeling process on random planar maps coupled to an {$O (n)$} loop model (with an appendix by {Linxiao Chen})}, arXiv preprint arXiv:1809.02012,  (2018).

\bibitem{Curien_Peeling_2023}
{\sc N.~Curien}, {\em Peeling random planar maps}, vol.~2335 of Lecture Notes in Mathematics, Springer, 2023.
\newblock \'{E}cole d'\'{E}t\'{e} de Probabilit\'{e}s de Saint-Flour XLIX---2019, \'{E}cole d'\'{E}t\'{e} de Probabilit\'{e}s de Saint-Flour. [Saint-Flour Probability Summer School].

\bibitem{Duchi_Fighting_2017}
{\sc E.~Duchi, V.~Guerrini, S.~Rinaldi, and G.~Schaeffer}, {\em Fighting fish}, J. Phys. A, 50 (2017), pp.~024002, 16.

\bibitem{Duchi_Fighting_2017a}
{\sc E.~Duchi, V.~Guerrini, S.~Rinaldi, and G.~Schaeffer}, {\em Fighting fish: enumerative properties}, S\'{e}m. Lothar. Combin., 78B (2017), pp.~Art. 43, 12.

\bibitem{Duchi_bijection_2022}
{\sc E.~Duchi and C.~Henriet}, {\em A bijection between rooted planar maps and generalized fighting fish}, arXiv preprint arXiv:2210.16635,  (2022).

\bibitem{Duchi_bijection_2023}
\leavevmode\vrule height 2pt depth -1.6pt width 23pt, {\em A bijection between {T}amari intervals and extended fighting fish}, European J. Combin., 110 (2023), pp.~Paper No. 103698, 21.

\bibitem{ElveyPrice_six_2023}
{\sc A.~Elvey~Price and P.~Zinn-Justin}, {\em The six-vertex model on random planar maps revisited}, J. Combin. Theory Ser. A, 196 (2023), pp.~Paper No. 105739, 32.

\bibitem{Evans_Combinatorial_2023}
{\sc P.~Evans and C.~Wenk}, {\em Combinatorial properties of self-overlapping curves and interior boundaries}, Discrete Comput. Geom., 69 (2023), pp.~91--122.

\bibitem{Fang_Fighting_2018}
{\sc W.~Fang}, {\em Fighting fish and two-stack sortable permutations}, S\'{e}m. Lothar. Combin., 80B (2018), pp.~Art. 7, 12.

\bibitem{ferrari2024randomdisksconstantcurvature}
{\sc F.~Ferrari}, {\em Random disks of constant curvature: the lattice story}, arXiv preprint arXiv:2406.06875,  (2024).

\bibitem{ferrari2025jackiw}
{\sc F.~Ferrari}, {\em {Jackiw-Teitelboim} gravity, random disks of constant curvature, self-overlapping curves, and {Liouville CFT$_1$}}, Physical Review D, 111 (2025), p.~L061901.

\bibitem{Godet_New_2020}
{\sc V.~Godet and C.~Marteau}, {\em New boundary conditions for {$\rm AdS_2$}}, J. High Energy Phys.,  (2020), pp.~Paper No. 020, 64.

\bibitem{Graver_When_2011}
{\sc J.~E. Graver and G.~T. Cargo}, {\em When does a curve bound a distorted disk?}, SIAM J. Discrete Math., 25 (2011), pp.~280--305.

\bibitem{Gwynne_Mating_2023}
{\sc E.~Gwynne, N.~Holden, and X.~Sun}, {\em Mating of trees for random planar maps and {L}iouville quantum gravity: a survey}, in Topics in statistical mechanics, vol.~59 of Panor. Synth\`eses, Soc. Math. France, Paris, 2023, pp.~41--120.

\bibitem{Kammerer_Gaskets_2024}
{\sc E.~Kammerer}, {\em Gaskets of {$O (2)$} loop-decorated random planar maps}, arXiv preprint arXiv:2411.05541,  (2024).

\bibitem{Kazakov_Disc_2022}
{\sc V.~Kazakov and F.~Levkovich-Maslyuk}, {\em Disc partition function of 2d {$R^2$} gravity from {DWG} matrix model}, J. High Energy Phys.,  (2022), pp.~Paper No. 190, 39.

\bibitem{Kazakov_Almost_1996}
{\sc V.~A. Kazakov, M.~Staudacher, and T.~Wynter}, {\em Almost flat planar diagrams}, Comm. Math. Phys., 179 (1996), pp.~235--256.

\bibitem{LeGall_Uniqueness_2013}
{\sc J.-F. Le~Gall}, {\em Uniqueness and universality of the {B}rownian map}, Ann. Probab., 41 (2013), pp.~2880--2960.

\bibitem{Mertens_Solvable_2023}
{\sc T.~G. Mertens and G.~J. Turiaci}, {\em Solvable models of quantum black holes: a review on {Jackiw--Teitelboim} gravity}, Living Reviews in Relativity, 26 (2023), p.~4.

\bibitem{Mullin_enumeration_1967}
{\sc R.~C. Mullin}, {\em On the enumeration of tree-rooted maps}, Canadian J. Math., 19 (1967), pp.~174--183.

\bibitem{oeis}
{\sc {OEIS Foundation Inc.}}, {\em The {O}n-{L}ine {E}ncyclopedia of {I}nteger {S}equences}, 2025.
\newblock Published electronically at \url{http://oeis.org}.

\bibitem{Pardo_Square_2023}
{\sc A.~Pardo}, {\em Square-tiled tori}, Discrete Contin. Dyn. Syst., 43 (2023), pp.~1926--1941.

\bibitem{Poenaru_Extension_1995}
{\sc V.~Po\'{e}naru}, {\em Extension des immersions en codimension 1 (d'apr\`es {S}amuel {B}lank)}, in S\'{e}minaire {B}ourbaki, {V}ol. 10, Soc. Math. France, Paris, 1995, pp.~Exp. No. 342, 473--505.

\bibitem{Saad_JT_2019}
{\sc P.~Saad, S.~H. Shenker, and D.~Stanford}, {\em {JT} gravity as a matrix integral}, arXiv preprint arXiv:1903.11115,  (2019).

\bibitem{Sheffield_Quantum_2016}
{\sc S.~Sheffield}, {\em Quantum gravity and inventory accumulation}, Ann. Probab., 44 (2016), pp.~3804--3848.

\bibitem{Stanford_JT_2020}
{\sc D.~Stanford and E.~Witten}, {\em J{T} gravity and the ensembles of random matrix theory}, Adv. Theor. Math. Phys., 24 (2020), pp.~1475--1680.

\bibitem{Tutte_census_1963}
{\sc W.~T. Tutte}, {\em A census of planar maps}, Canadian Journal of Mathematics, 15 (1963), pp.~249--271.

\bibitem{Zonneveld_tree_2025}
{\sc B.~Zonneveld}, {\em A tree bijection for rigid quadrangulations}, arXiv preprint,  (2025).

\bibitem{Zorich_Flat_2006}
{\sc A.~Zorich}, {\em Flat surfaces}, in Frontiers in number theory, physics, and geometry. {I}, Springer, Berlin, 2006, pp.~437--583.

\end{thebibliography}
	
\end{document}